\documentclass[11pt,reqno]{amsart}
\usepackage{pkgfile}
\usepackage[usenames]{color}

\def\ind{{\mathbf 1}}
\def\l{\ell}
\newcommand{\bk}{\mathbf{k}}
\newcommand{\bg}{\mathbf{g}}
\newcommand{\ba}{\mathbf{a}}
\newcommand{\bs}{\mathbf{s}}

\begin{document}
\title[Hierarchical Coulomb Gas]{Ground states and hyperuniformity of the hierarchical coulomb gas in all dimensions}

\author[S. Ganguly]{Shirshendu Ganguly}
\address{S. Ganguly\\
  Department of Statistics\\
 U.C. Berkeley \\
  Evans Hall \\
  Berkeley, CA, 94720-3840 \\
  U.S.A.}
  \email{sganguly@berkeley.edu}

 \author[S. Sarkar]{Sourav Sarkar}
\address{S. Sarkar\\
 Department of Statistics\\
 U.C. Berkeley \\
  Evans Hall \\
  Berkeley, CA, 94720-3840 \\
  U.S.A.}
  \email{souravs@berkeley.edu}

  \date{}    
       
\date{}
\maketitle
           
\begin{abstract}
Stochastic point processes with Coulomb interactions arise in various natural examples of statistical mechanics, random
matrices and optimization problems. Often such systems due to their natural repulsion exhibit remarkable  hyperuniformity properties, that is, the number of points landing in any given region fluctuates at a much smaller scale compared to that of a set of i.i.d. random points. 
A well known conjecture from physics appearing in the works of Jancovici, Lebowitz, Manificat, Martin, and Yalcin (see \cite{lebowitz83, martinyalcin80, jancovicietal93}), states that the variance of the number of points landing in a set should grow like the surface area instead of the volume unlike i.i.d. random points. 
In a recent beautiful work \cite{cha}, Chatterjee gave the first proof of such a result in dimension three for a Coulomb type system, known as the hierarchical Coulomb gas, inspired by {Dyson's} hierarchical model of the Ising ferromagnet \cite{dyson53,dyson69}. However the case of dimensions greater than three had remained open. 
In this paper, we establish the correct fluctuation behavior up to logarithmic factors in all dimensions greater than three, for the hierarchical model. Using similar methods, we also prove sharp variance bounds for smooth linear statistics which were unknown in any dimension bigger than two. A key intermediate step is to obtain precise results about the ground states of such models whose behavior can be interpreted as hierarchical analogues of various crystalline conjectures predicted for energy minimizing  systems, and could be of independent interest. 
\end{abstract} 
\tableofcontents      
\section{Introduction}
In this article, we consider certain finite particle models of the Coulomb Gas which in general can be studied in the framework of  Gibbs measures on point processes (finite or infinite)  interacting in a potential.   
Informally, this is a model of $n$ particle interacting gas in an ambient space such as $\R^d$ or $[0,1]^d,$ which is a probability measure on points $(x_1,x_2,\ldots,x_n)$ having a joint density,

\begin{equation}\label{potential0}
\frac{1}{Z}\exp\left(-\beta \sum_{i\neq j}w(x_i,x_j)-\beta n \sum_{i=1}^nV(x_i)\right)
\end{equation}
where $w(x, y)$ is a symmetric interaction term, $V$ is an external potential, $\beta$ is some positive parameter and $Z=Z(\beta,n)$ is the partition function or normalizing constant making the above a probability measure. Throughout this article we will focus on finite point processes even though there is a rich literature devoted to infinite systems which are often defined as a limit of their finite counterparts.
The standard Coulomb interaction has the following dimension dependence:
\begin{equation}
w(x,y)=\left \{ \begin{array}{cc}
|x-y|& \,d=1,\\
-\log |x-y|& \,d=2,\\
|x-y|^{2-d}& \,d\ge 3,\\
\end{array}
\right.
\end{equation}
and additionally $V(x)$ is taken to be $|x|^2.$
Before describing the precise setting and results in this paper, we start by briefly reviewing the rich history of the study of Coulomb systems. This account will be far from complete and we refer the interested reader to \cite{serfatyICM} and the references therein for a detailed account of the various  recent progress made in different fronts. 
Classical Coulomb systems are ubiquitous statistical mechanics models. The one-dimensional Coulomb gas is an integrable model, while  
the two-dimensional Coulomb gas has several connections to random matrices; for e.g., the $\beta = 1$ case is related to the Ginibre ensemble; similarly there are connections between the so called one-dimensional `log-gases'  and the classical Gaussian matrix ensembles,   (see \cite{forrester10}, and \cite{agz10}). 
For general $\beta$, even though  the two-dimensional Coulomb gas is not exactly solvable, there has still been a lot of progress: see for e.g. \cite{bz98, petzhiai98, hardy12} for large deviation principles, which were extended to higher dimensions in  \cite{chafaietal14, serfaty14}, while concentration inequalities were proved in \cite{chafai16}. However in  general many things remain rather unclear in dimensions bigger than two.
\subsection{Fluctuation Theory}\label{fluc1}
A fundamental statistic of interest is the fluctuation of the number of particles  $N(U)$ landing in a domain $U$. A first order of business is often to estimate the number variance, $\Var(N(U))$. As in Poisson processes, in many natural examples, one has $\Var(N(U))$ growing like the volume of $U$.  In such cases, the model is said to be extensive.
However, many models that exhibit some form of repulsion, turn out to be very rigid, or as they are called in the physics literature, hyperuniform or superhomogeneous; where  the growth of fluctuations is much slower. In fact, in such cases it is often predicted to grow like the surface area of $U$. Such situations are termed sub-extensive. 
Many one and two-dimensional examples have been mathematically shown to exhibit hyperuniformity. For instance, \cite{costinlebowitz95, diaconisevans01} established the phenomenon for random unitary matrices; while random
hermitian matrices were treated  in \cite{pastur06, bey12,bey14,bey14b, ghosh15, taovu13, erdhos2012}. The spectrum of 
non-hermitian random matrices are also known to exhibit rigidity: see for e.g., \cite{borodinsinclair09,
byy14a, byy14b, ghoshperes17}.
Furthermore, 
random analytic functions  in this context has been analyzed to great detail in \cite{ houghetal09, nazarovsodin11, ghosh16, gz16, gl17,
ghoshperes17}. For a model of a different nature, \cite{peressly14, holroydsoo13},  investigated the related question of deletion tolerance for the models of a perturbed lattice with Gaussian noise.

\subsubsection {Rigidity for Coulomb type systems}
Hyperuniformity of the two-dimensional
Coulomb gas has recently been established in \cite{bbny15, bbny16, lebleserfaty16}.
But as far as we know, before Chatterjee's results, no such results establishing rigidity for systems in dimensions three and higher were available,  
although very precise information about the partition function in such contexts had recently been obtained in \cite{rougerieserfaty16, lebleserfaty15}.
Nonetheless, there is a considerable physics literature on this
topic. We point the reader to the recent survey
of Ghosh and Lebowitz \cite{gl17b}, for a detailed account. In particular, as indicated previously, \cite{martinyalcin80, martin88, lebowitz83} suggest that for any domain  $U,$ the fluctuation of the number of particles  i.e.,  $N(U)-\E(N(U))$ grows like $\sqrt{|\partial{U}|}$ and hence the model is sub-extensive. Moreover,
\begin{equation}\label{conj21}
\frac{N(U)-\E(N(U))}{\sqrt{|\partial{U}|}}
\end{equation}
is predicted to converge to a Gaussian distribution as $|U|\to \infty$ along dilations of a fixed region, say $U_0.$  

\subsection{Ground states}\label{ground1}
Another major area of research is the study of zero temperature version of the models discussed in the previous section, i.e.,  when one takes $\beta=\infty$ in \eqref{potential0}. In that case the measure is supported on the so called ground states or energy minimizing configurations. It is predicted that in many situations, the minimizing energy configurations have a crystalline/lattice  structure. However rigorous results are only known for very special examples which do not include Coulomb type interactions, see  for e.g., \cite{The, BPT, HR, Sut, radin81}. Often such questions also have a number theoretic flavor; in particular the planar case is connected to the so called Epstein-Zeta function. However, the picture in high dimensions remains far from being understood. See \cite{rougerieserfaty16} for an elaboration of the above points.
Also, very recently there has been some work in  \cite{BeltranHardy} towards the well known Smale's seventh problem \cite{Smale}, where the authors approximate the ground state in a two dimensional Coulomb system of $N$ particles by considering a typical configuration at low temperatures. More precisely, \cite{BeltranHardy} shows that when the temperature is  $O(\frac{1}{N})$, with high probability the interaction energy of a typical configuration will be within $O(\log(N))$ of the ground state energy.

\subsection{Our contributions }
The starting point of this paper is the recent beautiful work of  Chatterjee  \cite{cha} who considered the so called Hierarchical Coulomb gas model and established the correct fluctuation behavior  up to logarithmic corrections in dimensions $3,$ matching the previously alluded to prediction, in \cite{lebowitz83, martinyalcin80, jancovicietal93} (The arguments work in dimensions $1$ and $2$ as well). 
Building on this, the goal of this paper is two fold:
\begin{enumerate}
\item Have a refined understanding of the ground state configurations for the Hierarchical model and their energy. They turn out to have certain interesting number theoretic features which allow us to obtain precise estimates. 
\item Extend Chatterjee's result by establishing the correct fluctuation behavior for the number of points landing in a domain as well as smooth linear statistics in any dimension. The proof of this part uses crucially the results obtained about the ground states. 
\end{enumerate} 
After having stated the precise definitions and the main results  in the next section, in Section \ref{iop} we will describe in some detail Chatterjee's method and the new techniques needed to go beyond, to  prove the main results in this paper.

\section{Model definitions and Main results}\label{mdmr}
We start by specializing the general definition stated in \eqref{potential0} to our case. We will be considering a  probability measure on points $(x_1,x_2,\ldots,x_n)$ in $[0,1)^d,$ by taking $V=0 $ on the hypercube and $\infty$ outside. 
Thus the density in \eqref{potential0}  reduces to
\begin{equation}\label{hierar1}
\frac{1}{Z}\exp\left(-\beta \sum_{i\neq j}w(x_i,x_j)\right).
\end{equation}
We now have the following description of the hierarchical interaction $w(\cdot,\cdot),$ taken from \cite{cha} (we will only describe the regime $d\ge 3$). Consider the tree of dyadic sub-cubes by sub-dividing the unit cube $[0,1)^d$  into $2^d$ sub-cubes of side-length
$1/2$, and iteratively sub-dividing each of them.  For
any pair of distinct points $x$ and $y$ in $[0,1)^d$, we define $$w(x, y) = 2^{(d-2){(k-1)}},$$ where
$k$ is the smallest integer such that $x$ and $y$ do not belong to the same dyadic cube
of side-length $2^{-k}$. Thus the minimum potential between two points in this definition is $1$. It is worth pointing out that in \cite{cha}, $w(x, y)$ was defined to be  $2^{(d-2){k}}$ which is a constant multiple of our definition, but it would be convenient later for us to assume our setting.
Note that the above definition does not cover  points that lie on the
boundaries of the cubes. However we choose to ignore such measure zero set of points.
Also observe that $w$ is symmetric but clearly not translation invariant.
Finally we record the expression for the partition function which will be one of the central objects we will analyze en route our results.
\begin{equation}\label{partition}
Z_n=Z(n,\beta)=\int\ldots\int \exp\left(-\beta \sum_{i\neq j}w(x_i,x_j)\right) {\rm d}x_1 {\rm d}x_2 \ldots {\rm d}x_n \,.
\end{equation}
Before proceeding further, it would be convenient to define some notations.
\subsection{Notations}
For brevity, we will often use $\bx$ to denote the tuple $(x_1,x_2, \ldots,x_n)$ and $H_n(\bx)=\sum_{i\neq j} w(x_i,x_j),$ to denote the Hamiltonian. 
 As already mentioned, we will work with the semi-open unit cube $[0,1)^d$.  Recall that a  dyadic sub-interval of $[0, 1)$ is of the form $[\frac{i}{2^{k}}, \frac{i+1}{2^{k}})$, where $k\ge 0$ and  $0 \le i \le 2^k - 1$ and  a dyadic sub-cube of $[0, 1)^d$ is of the form $I_1 \times I_2\times I_3 \times \ldots \times I_d$, where $I_1, I_2,\ldots,I_d$ are dyadic sub-intervals of $[0, 1)$ of length $\frac{1}{2^k}$ for some $k\ge 0$.
Thus let $$\cD_{k}=\left\{\left[\frac{i_1}{2^{k}}, \frac{i_1+1}{2^{k}}\right)\times \ldots \times \left[\frac{i_d}{2^{k}}, \frac{i_d+1}{2^{k}}\right): 0\le i_1,i_2,\ldots i_d \le 2^{k}-1\right\}$$
 be the set of all dyadic sub-cubes of $[0,1)^d$ of side length $2^{-k}$, and let
\begin{equation}\label{tree}
\cD : =  \bigcup_{k=0}^\infty \cD_k,
\end{equation}
denote the set of all dyadic sub-cubes of $[0,1)^d$. Not surprisingly, it would be important for us to exploit the natural tree structure on  $\cD$, with each node having $2^d$ children.  We will  often use the terminology typically associated to describing a tree such as child, parent, ancestor, descendant, generation/level etc. to denote the obvious corresponding objects in $\cD$.

Hierarchical models were introduced by Dyson in statistical mechanics to study a variant of the one-dimensional Ising model \cite{dyson53, dyson69}. Subsequently, the so called \textit{Dyson's hierarchical model} has led to a lot of activity, see for e.g., \cite{benfattoetal86, marchettiperez89, dimock90,  kappeleretal91, benfattorenn92, guidimarchetti01} for results about the two-dimensional hierarchical Coulomb gas. However, before \cite{cha} not much was known about this model in any dimension greater than $2$.

\subsection{Statements of the theorems}
To state the main results, we first set things up following some of the definitions in \cite{cha}.
For any $d\geq 3$, let $U$ be a nonempty open subset of $\R^d$ and $\partial U$ denote the boundary of $U$. For every $\varepsilon>0$, we will denote the set of all points, at a distance of at most $\varepsilon$ from $\partial U,$ by 
$\partial U_\varepsilon$, and the diameter of $U$ will be called $\mathrm{diam}(U)$. The boundary of $U$ will be said to be \textit{regular} if there exists some constant $C$ such that for all $0<\varepsilon\leq \mathrm{diam}(U)$,
\[\mathrm{Leb}(\partial U_\varepsilon)\leq C\varepsilon\,,\]
where $\mathrm{Leb}$ denotes the Lebesgue measure. Note that if the boundary is a smooth, closed, orientable surface, then it is regular. Finally  as mentioned before, $N(U)$ will be used to denote the number of points in $U.$ We are now in a position to state our main result which proves the correct variance bound for $N(U)$ up to logarithmic factors under the above stated regularity properties.

\begin{maintheorem}[Macroscopic Hyperuniformity]\label{marigid} Let $U$ be a non empty connected open subset of $[0,1)^d$ with a boundary that is a smooth, closed, orientable surface.  Then there exists positive constants $C(U,\beta), c(U,\beta)$ such that, 
\[c(U,\beta) n^{\frac{d-1}{d}}\le \mbox{Var}(N(U))\leq C(U,\beta)n^{\frac{d-1}{d}}\log^{13} n\,.\] 
The lower bound, is in fact a corollary of the following stronger fact about anti-concentration: 
 there are constants $ c > 0$ and $p < 1$, depending only on $U$ and $\beta$, such that for all large enough $n$ we have $$\P(a\le N(U)\le b)\le p,$$ for any $a \le b \in \R$ with $b-a\le cn^{\frac{d-1}{2d}}$. \end{maintheorem}
As discussed around \eqref{conj21}, the logarithmic factors in the upper bound are artifacts of the proof and are expected to be superfluous.  
We next state a key property of the hierarchical model which states that the dyadic boxes exhibit much stronger hyperuniformity. This is a significant improvement over the corresponding result in \cite{cha} which for $d=3,$ proves a $O(n^{\frac23})$ bound on the same. 
\begin{maintheorem}[Strong rigidity estimates for dyadic cubes]\label{keyrigid} For any $D\in \cD_1,$
\[\mbox{Var}(N(D))\leq K(\beta)\max\{\log^{12}(n),1\}\,,\]
where $K(\beta)$ is a non-increasing function of $\beta$. Moreover the hierarchical structure of the model and that $K(\beta)$ is non-increasing, implies the same bound for $\Var(N(D))$ for any dyadic box $D\in \cD.$ 
\end{maintheorem}

A crucial ingredient in the proofs of the above theorems is the following estimate on the ground state energy. Let $$L_{n}:=\min_{\bx}H_n(\bx),$$ for $n\geq 2$, where the infimum is taken over point processes which do not intersect the boundaries of the sub-cubes. 

\begin{maintheorem}[Ground state energy]\label{gse}For any $m\in \N $, let
$$m=\sum_{i=0}^{k_m}c^{(m)}_i2^{di}\,,$$ where  $0\le c^{(m)}_i\le 2^{d}-1$ for all $i\le k_m$ and $c^{(m)}_{k_m}>0$ be the representation of $m$ in base $2^d.$ Then,
\[L_{n}=\frac{(C_d+2)(n-1)n}{2}-C_d\sum_{m=1}^{n-1}[\sum_{i=0}^{k_m}c^{(m)}_i2^{(d-2)i}]
\,,\]
where
$C_d:=\frac{2^{d-1}-2}{3\cdot 2^{d-2}}\,.$ 
\end{maintheorem}
\begin{rmk} Note that the above implies in particular that $L_n=\frac{(C_d+2)n^{2}}{2}-O(n^{\frac{d-2}{d}+1})$, which already follows from Chatterjee's method to bound $L_n$ (see \cite[Theorem 2.2]{cha}) based on an elegant application of the Cauchy-Schwarz inequality. A refined statement of a similar flavor for the standard Coulomb gas model was obtained in the seminal work \cite{rougerieserfaty16}.
\end{rmk}

Our next result gives a description of the ground states. As mentioned in the introduction, there has been a lot of activity around identifying energy minimizing configurations which for Coulomb systems are predicted to exhibit certain crystalline or lattice structures. The next result can be thought of as an analogue of such statements in the hierarchical setting where the problem is much easier to analyze. 

\begin{maintheorem}[Ground state configurations]\label{gsc} A configuration $\sigma$ is a ground state iff for any vertex $D\in \cD$, if $D_{1}, D_2,\ldots D_{2^d}$ denote the children of $v$ then the following two conditions are satisfied:
\begin{enumerate}
\item $N(D)=\sum_{i=1}^{2^{d}}N(D_i).$
\item $|N(D_i)-N(D_j)|\le 1$ for all $i,j.$
\end{enumerate}
\end{maintheorem}

Note that the first condition is implicit in the definition of a configuration. So the only non-trivial constraint says that, a configuration is a ground state if and only if inside any dyadic box, the particles are partitioned as uniformly as possible among the dyadic boxes of the next level. 

Finally we state our results about linear statistics. For any function $f:[0,1]^d\to \R,$ the associated linear statistic is the following: 

\begin{equation}\label{linstat}
X(f)=\sum_{i=1}^{n}f(X_i),  
\end{equation}
where $(X_1,\ldots,X_n)$ forms a Coulomb system of size $n.$ Thus, as in the previous theorems $N(U)$ is obtained when $f$ is just the indicator of $U.$ But in what follows we will be interested in smooth functions.  Fluctuation theory for smooth linear statistics of point processes has a rich history.  In dimensions one and two, which in particular covers the case of eigenvalues of random matrices, it is known that $X(f)-\E(X(f))$ is a tight random variable, unlike sum of i.i.d. random variables where the variance grows like ${n}.$ Furthermore in many cases, one can also show a central limit theorem proving that $X(f)-\E(X(f))$ converges to a Gaussian distribution. See for e.g., \cite{bbny15, bbny16, bey12, bey14, bey14b, diaconisevans01, joh19, joh20, lebleserfaty16}.
However in dimensions bigger than two, there is a change in behavior of $X(f)$ and the variance grows polynomially in the system size.  Chatterjee had analyzed the case of three dimensions and had proved the following non-matching upper and lower bounds: 
\begin{equation}\label{chasmooth}
n^{1/3}\le \Var(X(f))\le n^{2/3}.
\end{equation}
 It was not clear what the sharp bound would be and in the vaguely related context of orthogonal polynomials, \cite{bardenethardy16} had established a sharp variance exponent of $2/3$ providing some support that the upper bound was correct. 
In the following theorem we establish the correct variance bound in any dimension, in particular proving that in fact the lower bound in \eqref{chasmooth} is sharp. 
\begin{maintheorem}[Smooth linear statistics]\label{smoothsharp} Assuming $f$ as above is a Lipschitz function with Lipschitz constant $L$, then $\Var(X(f))= O(K(\beta)L^2)n^{\frac{d-2}{d}},$
where $K(\beta)$ is the constant appearing in Theorem \ref{keyrigid}.
\end{maintheorem}
A matching lower bound of $\Theta(n^{\frac{d-2}{d}})$ for linear functions can be obtained by following the argument of the lower bound in dimension three from \cite[Theorem 1.5]{cha}. Note that unlike Theorem \ref{marigid}, Theorem \ref{smoothsharp} does not lose any logarithmic factors.

\section{Ideas of the proofs}\label{iop}
The advantage of working with the hierarchical model is that, often a necessary coarse graining of the potential field is inbuilt into the definition.
Thus under this setting,  a natural first choice of the set $U$ in the context of the discussion around \eqref{conj21}, would be a dyadic box of side length $1/2$.  Now note that in dimension $d$ there are $2^d$ such boxes. 
Let $\cP$ be a random configuration of $n$ points in $[0,1]^d$ sampled according to the hierarchical Coulomb gas measure and for any dyadic box $v$  recall that $N(v)$  denotes the number of points landing in $v.$

We will now specialize to the case $d=3$ to  describe briefly the approach in Chatterjee's work and then the new ideas needed to go beyond dimension $3$. Let $v_1,v_2,\ldots, v_{8}$ denote all the eight dyadic boxes of side length $1/2.$ 
Since the expected number of particles in each $v_i$ by symmetry is $\frac{n}{8}$, it is natural to compute the probability of the event 
\begin{equation}\label{discrepancy}
N(v_i)=\frac{n}{8}+k_i,
\end{equation}
for $i=1,2,\ldots,8$, for a given vector of integers $(k_1,k_2,\ldots,k_8)$ with $\sum_{i=1}^{8} k_i=0$ (we assume, for simplicity, that $n/8$ is an integer).
Now it is not too hard to observe that, 
\begin{align}\label{probest1}\P\left(N(v_i)=n_i \text { for }i=1,2,\ldots,8,\right)= 8^{-n}\frac{n!}{n_1!n_2!\ldots n_8!}e^{-\beta\sum_{1\leq i\neq j\leq 8}n_in_j} \frac{\prod_{i=1}^8Z(2\beta,n_i)}{Z(\beta,n)}.
\end{align}
The $8^{-n}$ factor accounts for the fact that the volume of each dyadic box is $\frac{1}{8}.$ The  multinomial factor comes from choosing the labels of the points landing in the boxes $v_1,v_2,\ldots,v_8.$ Now given these points, the Hamiltonian or the interaction energy comprises of two terms: $$e^{-\beta\sum_{1\leq i\neq j\leq 8}n_in_j}$$ is  the contribution of the interaction across different boxes. The remaining effect from both the interaction inside the dyadic boxes, as well as the entropy of the precise location of the points inside them appear as the product of the partition functions $\prod_{i=1}^8Z(2\beta,n_i),$ since the hierarchical nature of the model implies that the particle systems restricted to each of the cubes are themselves independent hierarchical coulomb systems with $\beta$ replaced by $2\beta$ ($2^{d-2}\beta$ in general).

Thus, pretending for the purpose of the sketch that $n$ is a multiple of $8$, the key step in \cite{cha} is to compare $\P\left(N(v_i)=n_i \text { for }i=1,2,\ldots,8,\right)$ with $\P\left(N(v_i)=\frac{n}{8} \text { for } i=1,2,\ldots,8\right)$.  Hence using \eqref{probest1}, it becomes crucial to obtain sharp estimates of the following types of ratio of partition functions, 
\begin{equation}\label{iter1}
\frac{Z(\beta,m+k)}{Z(\beta,m)}
\end{equation} for some integers $m, k$ and any $\beta>0.$
However it suffices to understand $\frac{Z(\beta,m\pm1)}{Z(\beta,m)},$ since, one can iterate it $k$ times to obtain the sought estimate. This is similar in flavor to the so called Cavity method in Spin Glasses where one adds one more spin to the system and analyzes its effects.
We will now briefly outline Chatterjee's argument to upper and lower bound $\frac{Z(\beta,m+1)}{Z(\beta,m)}.$ 
Start by observing that 
\begin{align}\label{ratio1}
\frac{Z(\beta,m+1)}{Z(\beta,m)}=\E\left(e^{-2\beta \sum_{j=1}^{m}w(X_i,U)}\right)
\end{align}
where $(X_1,\ldots, X_m)$ forms a Coulomb system of size $m$ and $U$ is an independent uniformly distributed over $[0,1]^3$ random variable. 
Now by Jensen's inequality,
$$\E\left(e^{-2\beta \sum_{j=1}^{m}w(X_i,U)}\right)\ge e^{-2\beta \sum_{j=1}^{m}\E w(X_i,U)}=e^{-2\beta \alpha m },$$ where $\alpha=\int \int w(x,y)dxdy,$ since for each $i$ the variable $X_i$ is uniformly distributed as well and independent of $U.$
Now to prove an upper bound on $\frac{Z(\beta,m+1)}{Z(\beta,m)}$, Chatterjee proves a lower bound on $\frac{Z(\beta,m)}{Z(\beta,m+1)}$ instead. Again similar to \eqref{ratio1} one can observe that \begin{align}\label{ratio2}
\frac{Z(\beta,m)}{Z(\beta,m+1)}=\E\left(e^{2\beta \sum_{j=1}^{m}w(X_i,X_{m+1})}\right)&\ge \exp\left({2\beta \sum_{j=1}^{m}\E w(X_i,X_{m+1})}\right)\\
& = \exp\left({\frac{\beta m}{{{m+1}\choose{2}}}\sum_{1\le i\neq j\le m+1}\E w(X_i,X_j)}\right)
\end{align}
where $(X_1,\ldots, X_{m+1})$ forms a Coulomb system of size $m+1.$ Note that the last equality is just a consequence of the exchangeability of the sequence $(X_{1},\ldots,X_{m+1}).$ Now the proof is complete by observing that $\sum_{1\le i\neq j\le m+1} w(X_i,X_j)\ge L_{m+1}$, where $L_{m+1}$ is the ground state energy. 
Thus one obtains, $$\frac{Z(\beta,m+1)}{Z(\beta,m)}\le e^{-\frac{2\beta}{m+1} L_{m+1}}.$$
At this point, Chatterjee using the Cauchy-Schwarz inequality, proves:  
\begin{equation}\label{groundstate1}
L_m=\alpha m(m+1) -O(m^{4/3}),
\end{equation}
using which one obtains the following estimate: 
 \begin{equation}\label{cavity1}
e^{-2\beta \alpha m}\le \frac{Z(m+1)}{Z(m)}\le e^{-2\beta \alpha m+O(m^{1/3})}.
\end{equation}
Iterating the above yields :
\begin{equation}\label{tail14}
\prod_{i=1}^8\frac{Z(\beta,m+k_i)}{Z(\beta,m)}\le e^{-\beta C \sum_{i=1}^8k_i^2+ \sum_{i=1}^8O(|k_i|)n^{1/3}}.
\end{equation}
for some constant $C$.
It turns out that the above bound holds even after taking into account the remaining factors in \eqref{probest1}. Hence  \eqref{tail14} shows that the probability that any $k_i\ge Cn^{1/3}$ for a large enough constant $C,$  decays like $e^{-\Theta(n^{2/3})}.$  This proves the upper bound. We will skip the argument for the lower bound for now. 
However in dimension $4,$ instead of \eqref{groundstate1}, one has $L_{m}=\alpha m^2 -O(m^{3/2}).$ Thus applying the same argument one obtains the bound 
\begin{equation}\label{tail16}
\prod_{i=1}^{16}\frac{Z(\beta,m+k_i)}{Z(\beta,m)}\le e^{-C\beta\sum_{i=1}^{16}k_i^2+ \sum_{i=1}^8O(|k_i|)n^{1/2}},
\end{equation}  which only implies a $O(\sqrt{n})$ bound on the fluctuation and this progressively becomes worse as the dimension increases. For e.g., this provides a bound of $n^{\frac35}\gg n^{\frac12}$ in dimension $5.$

To get around this, we take a more direct approach by first identifying exactly the ground state configurations and hence the precise value of $L_n$ for any $n$ as stated in Theorem \ref{gse}. As indicated in Section \ref{ground1}, there are several predictions stating that the energy minimizing configurations of the standard Coulomb system should look like certain lattices. We obtain an analogous statement for the hierarchical model as mentioned in Theorem \ref{gsc}. The proof relies on an inductive argument and local modifications to pass from a non-optimal configuration to another one where the energy strictly decreases. The argument is quite general and does not rely on the exact nature of the Coulomb potential. However for the Coulomb potential, 
the value of $L_n$ turns out to have a rather interesting  dimension dependent number theoretic description. Precise details are provided in Section \ref{ground10}. 

Given this sharp understanding, we then proceed to obtain  the following very precise estimate of the partition function $Z_{n}$: 
\begin{equation}\label{keyest1}
Z_n^{\rm Gr}\le Z_{n} \le Z_n^{\rm Gr} e^{\log^{6} n}\,,
\end{equation}
 where $Z_n^{\rm Gr}$ is the contribution to the partition function from all the ground state configurations. 
Our analysis of the ground states yields sharp estimates for $Z_n^{\rm Gr}$ which along with the above inequality,  imply similar bounds for $Z_n$. This is proved in Section \ref{s3}, relying on the key estimate recorded in Proposition \ref{l:entropy1}.

We end this section with a short sketch of the key idea in the proof of \eqref{keyest1}. Recall that to specify a configuration we have to specify the number of points landing in each dyadic sub-cube. 
Now for arbitrary configurations, it is quite hard to control the interaction between particles and thus estimate the Hamiltonian. But in order to exploit  our understanding of the ground state energies, we construct an interpolation or a chain/sequence of configurations between the ground states and a generic configuration in the following way: the configuration at step $i$ (say $\cP^{(i)}$) of the sequence has the property that the restriction to dyadic boxes in $\cD_i$ (of side length $2^{-i}$) are ground state configurations, and the configurations at  step $i+1$ of the sequence and those at step $i$ agree for the first $i$ levels (see {Figure} \ref{fig1} for an illustration).
\begin{figure}[t]
\centering
\includegraphics[width=.6\textwidth]{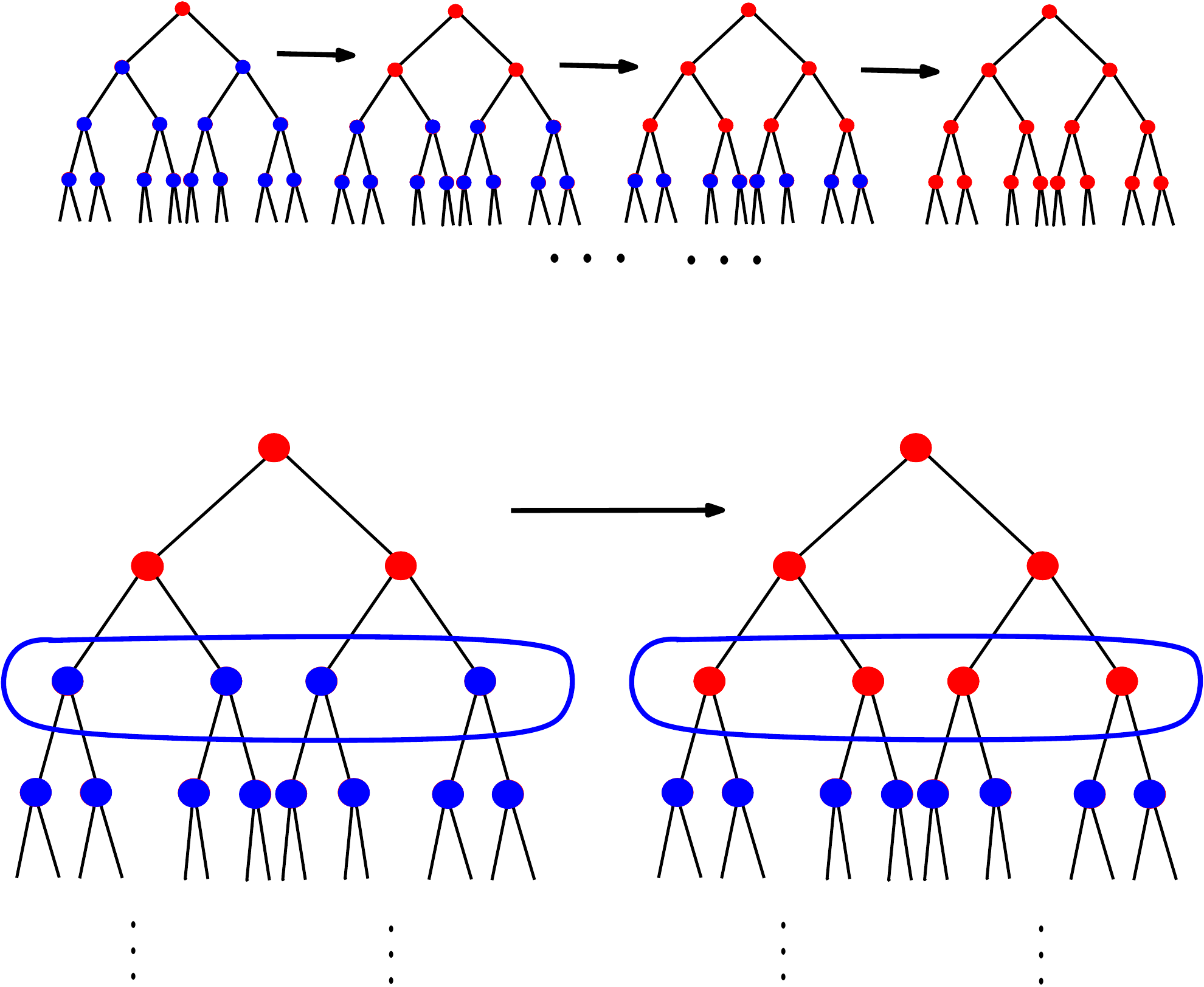}
\caption{Illustration of the proof strategy of interpolating between grounds states and any arbitrary state. The blue vertices denote ground state configurations while the red vertices denote arbitrary configurations.}
\label{fig1}
\end{figure}
We can now perform a refined energy vs entropy comparison at every step of this interpolating sequence. The argument consists of broadly three steps:
\begin{enumerate}
\item We first prove an estimate on the increase in energy i.e.,  $H(\cP^{(i+1)})-H(\cP^{(i)})$ as a function of an Euclidean type distance between the vectors say $a_i$ and $e_i$ denoting the partition of points  induced by $\cP^{(i)}$ and $\cP^{(i+1)}$ respectively at level $i+1$(on the elements of $\cD_{i+1}$). 
This is stated in  Lemma \ref{l:energy}. As a consequence,  if $i=\Omega(\log\log(n))$ then one obtains that 
\begin{equation}\label{prelim1}
Z(\cP^{(i+1)})\le e^{-c\beta 2^{(d-2)(i)} {\rm{dist}}(a_i,e_i)^2}Z(\cP^{(i)}),
\end{equation} for some constant $c,$ where $Z(\cP)$ denotes the contribution to the partition function by $\cP$. The precise formulation along with a formal definition of ${\rm{dist}}(\cdot,\cdot)$ appears in \eqref{e:imp2'}.

\item In the next step we show that under the assumption  $i= \Omega(\log\log n)$ the energy increase in step (1) recorded in \eqref{prelim1} is large enough to beat the  entropy of the number of possible distortions of the vector $a_i$ in to $e_i$ keeping  ${\rm{dist}}(a_i,e_i)$ fixed. Thus summing over all such distortions, one obtains the estimate 
\begin{equation}\label{ratioestimate}
\frac{Z_n^{(i+1)}}{Z_n^{(i)}} \le \Big(1+Ce^{-c_1\beta 2^{(d-2)i}}\Big),
\end{equation}
whenever $i\ge m_0,$
where $Z_n^{(i)}$ denotes the contribution to the partition function by configurations at step $i$ of the interpolation i.e., those whose restrictions to elements of $\cD_i$ form ground state configurations (thus $Z_n^{(1)}=Z_n^{\rm{Gr}}$) (this is precisely formulated in \eqref{e:imp2} and the discussion following it).
 We then use the telescopic product $$\displaystyle{\frac{Z_{n}}{Z_n^{\rm{Gr}}}=\prod_{i=1}^{\infty}\frac{Z_n^{(i+1)}}{Z_n^{(i)}}}$$ and the bound \eqref{ratioestimate}, to show that 
\begin{equation}
\label{key20}
Z_{n}-Z_n^{(m_0)}\leq C_0n^{-c_0}Z_n\,,
\end{equation}
for some constants $c_0,C_0>0,$ and $m_0=\Theta(\log\log n)$ (the precise value of $m_0$ will be specified later). 
We believe this approach of transferring estimates on ground state energies to partition functions by the method of interpolation could be useful in other related contexts. 
\item 
Equipped with these estimates we can now follow Chatterjee's approach closely to obtain 
${\rm{Var}}(N(D))=O(\log^{12}n)$  for any dyadic cube $D$. 
\end{enumerate}
\begin{figure}[t]
\centering
\includegraphics[width=.4\textwidth]{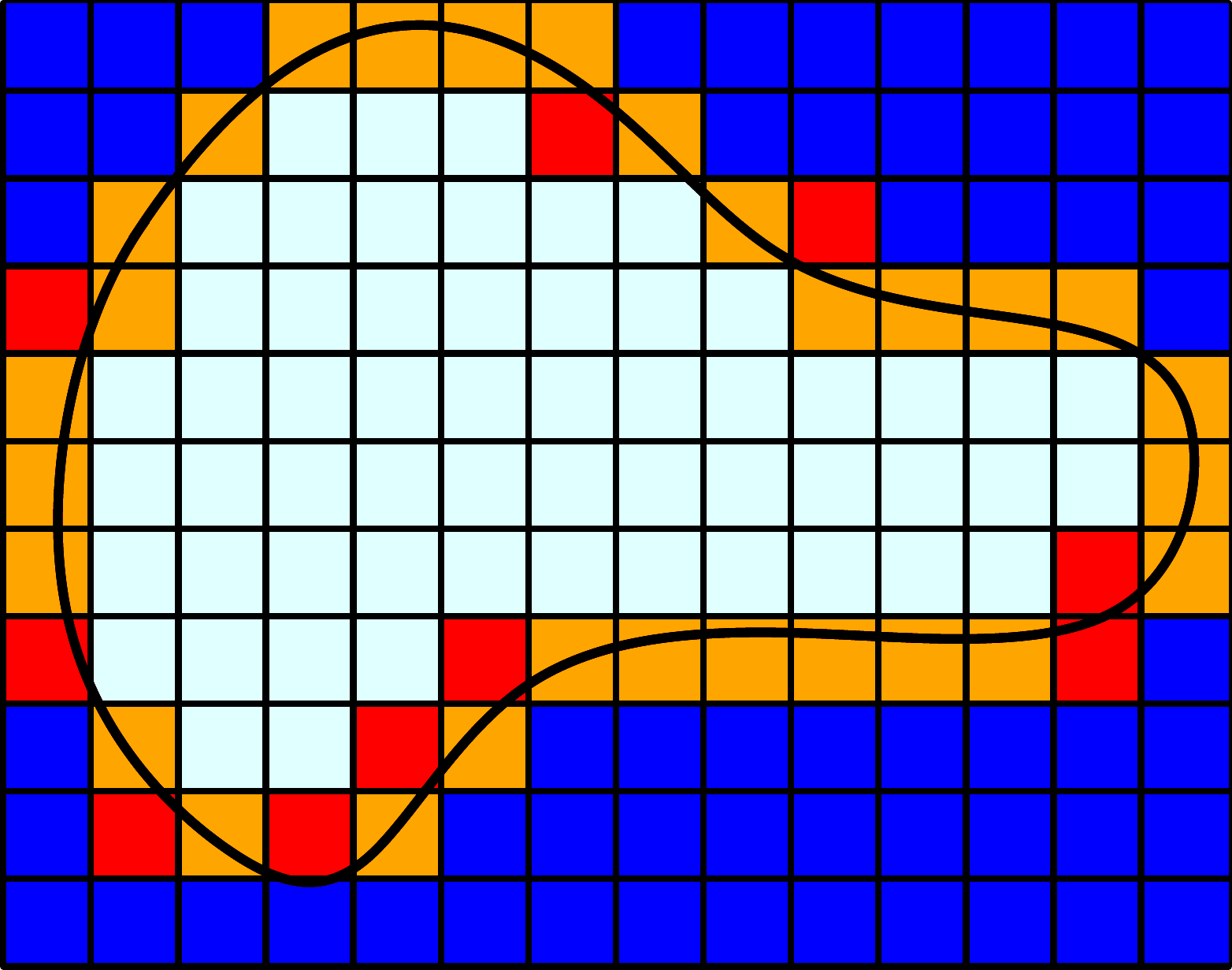}
\caption{Figure illustrating a sketch of the lower bound on fluctuations. The black curve denotes the boundary of the smooth domain $U.$ The entire cube is tessellated into dyadic cubes of size $n^{1/d},$ thus on average they have one point landing in them. The dark blue cubes lie outside $U$, the light blue cubes lie in the interior of $U$ while the yellow and red cubes lie on the boundary. The yellow ones additionally have the property that a significant fraction of their interior intersect both $U$ and $U^c.$ Smoothness of $U$ guarantees that the number of such cubes is $\Theta(n^{\frac{d-1}{d}}).$ Conditioning on the number of points landing in each of the cubes at this scale, on the typical event that a constant fraction of the yellow cubes have exactly one point in them, the number of points landing in $U$ is then up to a deterministic translation (function of the conditioning coming from points landing in the light blue cubes) a sum of $\Theta(n^{\frac{d-1}{d}})$ Bernoulli variables with probability of success bounded away from zero and one, leading to the sought lower bound on fluctuations. }
\label{fig5}
\end{figure}

At this point the reader might be wondering how the above is consistent with the conjecture recorded in  \eqref{conj21}.  It turns out that while for the actual coulomb potential, the conjecture is supposed to hold for every set $U,$ for the hierarchical model, the square root of the surface measure gives the right order of fluctuation only for ``generic" sets assumed in the hypothesis of Theorem \ref{marigid}. For a quick sketch of the idea behind the argument regarding the lower bound we refer the reader to Figure \ref{fig5}, leaving precise details for later. 
The proofs of Theorems \ref{marigid}  and \ref{smoothsharp} follows from Theorem \ref{keyrigid} by considering suitable martingales. 

\subsection{Further Remarks}There are several related fluctuation bounds besides Theorem \ref{marigid}, that one should be able to obtain using the methods in this paper. 
While this article includes results about the macroscopic behavior of the point process, one can pursue understanding the microscopic behavior as well by considering blowing up the system centered at a given point. We believe our methods should yield  rigidity estimates at such microscopical scales too. Furthermore, a closer look at the proof shows that the dependence on the exact nature of the Coulomb potential is not very important and we expect that one should be able to extend the methods of this paper to other singular interactions as well, such as the setting of Riesz gases. However such studies are not pursued here and will be taken up in future projects. 
The obvious open question is to obtain tight variance bound in Theorem \ref{marigid} by removing the log-discrepancies between the upper and lower bounds.
For a further refinement, we end this discussion by reiterating the open problem mentioned in \eqref{conj21} about proving a central limit theorem for $N(U)$ at scale $n^{\frac{d-1}{2d}}.$ 

\subsection{Organization of the article} In Section \ref{ground10} we prove the statements about the ground states (Theorem \ref{gse}, Theorem \ref{gsc}). Section \ref{pot1} is devoted to Theorem \ref{keyrigid} which contains most of the key new ideas in this paper, using heavily the results about the ground states obtained in Section \ref{ground10}. This section is rather long and for the ease of the reader a detailed roadmap is provided at the beginning of the section, indicating what the subsequent subsections achieve. Assuming Theorem \ref{keyrigid}, the proof Theorem \ref{marigid} is completed by following the arguments in \cite{cha} closely. For completeness, we sketch the key points in Section \ref{sketch}. Using similar ideas, the proof of Theorem \ref{smoothsharp} is provided in Section \ref{posmooth}.

\subsection{Acknowledgement} We thank Sourav Chatterjee for several  discussions  that inspired us to work on this problem. We also thank Paul Bourgade, Manjunath Krishnapur, and Sylvia Serfaty for many helpful comments. 

\section{Identifying ground states and the minimum energy}\label{ground10}

In this section, we prove Theorems \ref{gse} and \ref{gsc}. We will in fact  prove Theorem \ref{gsc} first.
We start by defining a point process  $\bz=(z_1,z_2,\ldots, z_n)$ such that $H(\bz):=H_n(\bz)=L_n.$
Note that by the hierarchical nature of the definition in \eqref{hierar1}, $H(\bz)$ is not quite a function of the points, but of the number of points landing in each dyadic cube. Hence it suffices to specify the partition $\cP$ (number of particle landing in each element of $\cD$) induced by $\bz.$  Subsequently, for any partition $\cP$ we will let $H(\cP):=H(\bz),$ for any point process $\bz$ corresponding to the partition $\cP.$

Recalling our convention (see \eqref{tree}) of identifying the dyadic cubes with the vertices of the $2^d-ary$ tree,   we will freely denote cubes by vertices of the tree and use standard terminology such as: two cubes $v$ and $w$  will be called siblings if the vertices have the same parent in the tree. Also, $\cP(v)$ will denote the number of points landing in $v$ under the partition $\cP$.

Let $\cP_{\min}:=\cP^{(n)}_{\min}$ be a partition which has the property that for any pair of siblings $v$ and $w$ (under the above mentioned identification) we have $$|\cP(v)-\cP(w)|\le 1.$$ Note that given $n,$ there is a unique such partition up to natural symmetries.  It is easy to define this inductively. Namely the following: 
Let $v_1,v_2, \ldots, v_{2^d}$ be adjacent to the root. Then 
\begin{equation}\label{induc0}
\cP_{\min}(v_i)=\left \{ \begin{array}{cc}
\lfloor \frac{n}{2^d}\rfloor+1& \,1\le i \le r,\\
\lfloor \frac{n}{2^d}\rfloor &   r< i.
\end{array}
\right.
\end{equation}
where $n=2^d\lfloor \frac{n}{2^d}\rfloor+r.$
Given the above, if $w_1,w_2,\ldots, w_{2^d}$ are the children of $v_j$ for any $1\le j \le 2^d$, then define  $\cP_{\min}(w_i)$ as in  \eqref{induc0} by replacing $n$ by  $\cP_{\min}(v_j).$
Given the above notations we now show that $H(\cP_{\min})=L_n$ to  prove Theorem \ref{gsc}.
\subsection{Proof of Theorem \ref{gsc}} The proof is by induction on $n$ and some local operations which modifies a partition not satisfying \eqref{induc0} to obtain another one with a smaller energy.   The first non-trivial case is $n=2$ where the theorem is easily seen to hold. 
Let us now prove it for $n$ assuming the statement for all $m<n.$
Consider a configuration $\cP_*$ such that the partition induced on the children $(v_1,v_2,\ldots,v_{2^d})$ of the root is $(a_1,a_2,\ldots,a_{2^d})$, such that $\sum_{i=1}^{2^d} a_i=n$ and $H(\cP_*)=L_n.$  

It is clear that $a_i$ is strictly less than $n$ for all $1\le i \le 2^d.$ 
Since otherwise if $a_1=n$ then the configuration  obtained by removing one point from the dyadic cube $v_1$ to the dyadic cube $v_2$ strictly decreases the energy and contradicts the hypothesis that $H(\cP_*)=L_n.$ 
Now since $a_i<n,$ for all $i,$ by induction hypothesis even if there could be multiple ground states, without loss of generality we can assume that $\cP_{*}\mid_{v_i}=\cP^{(a_i)}_{\min}$  defined in \eqref{induc0} with $n$ replaced by  $a_i$, and $\cP_{*}\mid_{v_i}$ denotes the partition $\cP_*$ restricted to the sub-tree rooted at $v_i$ (identified naturally with the whole tree).
Note that we will now be done once we show that $|a_i-a_j|\le 1$ for all $1\le i,j\le 2^d.$ We will prove this by contradiction.

Suppose not, and also without loss of generality using the underlying symmetry we can assume $a_1-a_2\ge 2.$
Now let $(w_1,w_2,\ldots, w_{2^d})$ and  $(z_1,z_2,\ldots, z_{2^d})$ be the children of $v_1$ and $v_2$ respectively. 
Let $\cP_{*}(w_i)=b_i$ and $\cP_{*}(z_i)=c_i.$ 
Thus 
\begin{equation}\label{part98}
\sum_{i=1}^{2^d} b_i=a_1, \text{ and similarly, }\sum_{i=1}^{2^d} c_i=a_2.
\end{equation}

Since we have assumed that $a_1- a_2\ge 2$ and induction hypothesis tells us the values of $b_i$ and $c_i$ up to symmetries,  it follows that we can choose the permutations of the $b_i's$ and $c_i's$ to ensure that,\begin{equation}\label{assum}
b_1-c_1\ge 1, b_2-c_2\ge 1 \text{ and } b_i\ge c_i  \text{ for all } 2<i\le 2^d.
\end{equation}

Note that the above says the following: Induction hypothesis allows us to conclude that $a_1- a_2\ge 2$ implies at least two $b_i$'s are bigger than the corresponding $c_i$'s.  
This is where we crucially use that the partition inside $v_1$ and $v_2$ are given by $\cP^{(a_1)}_{\min}$  and $\cP^{(a_2)}_{\min}$ respectively and their special structure.

We will now create a modified configuration $\cP'$ which will improve the energy.
The modified configuration is illustrated  in Figure \ref{mod} (for $d=2$).
\begin{figure}[t]
\centering
\includegraphics[width=.5\textwidth]{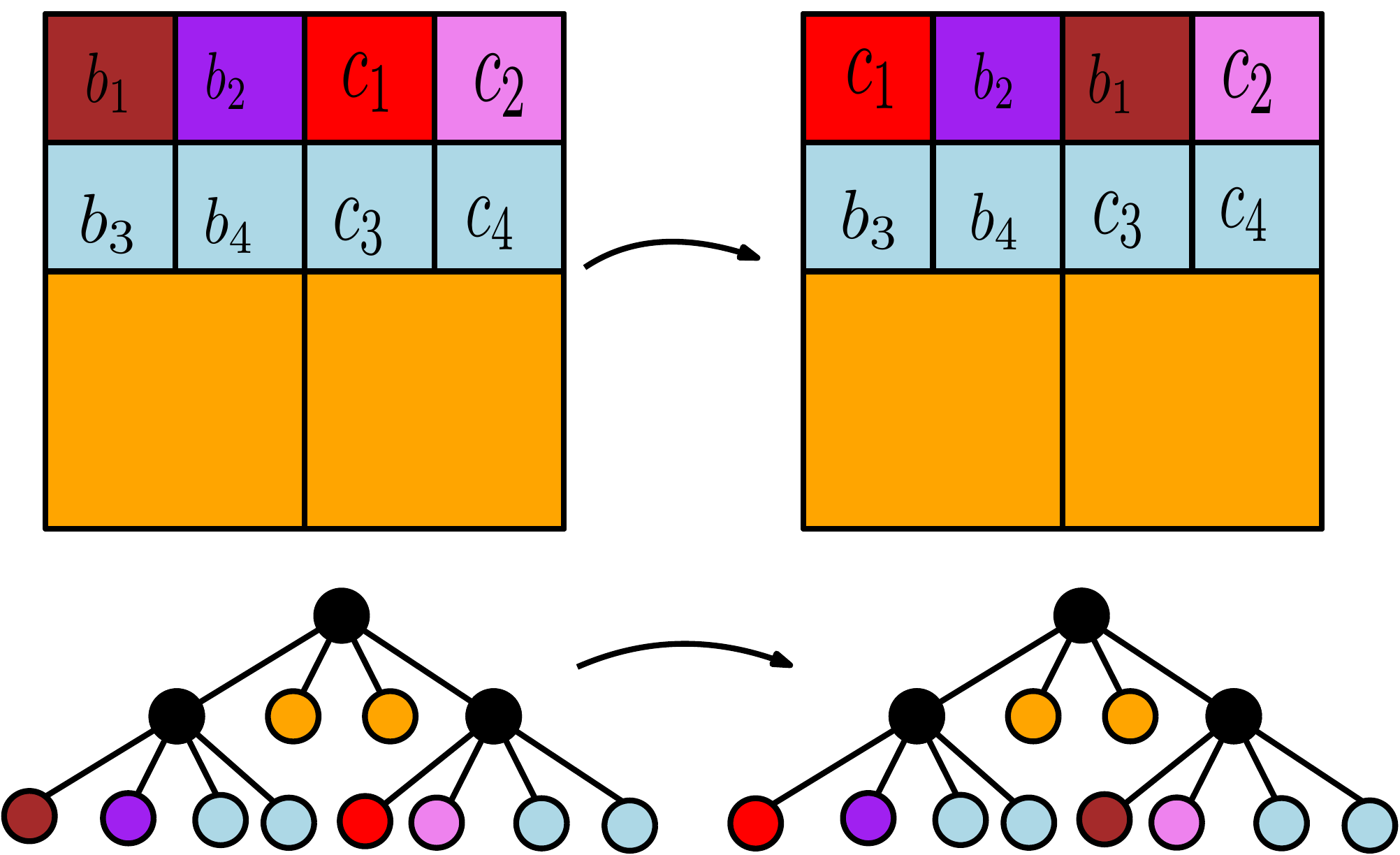}
\caption{The figure illustrates the modification scheme in the proof of Theorem \ref{gse}  in the planar case (although the proof is written for $d\ge 3$), where $b_1+b_2 +\ldots b_{4}\ge c_1+c_2+\ldots c_{4}+2$ and the $b_i's$ and the $c_i's$  themselves form ground state partitions by induction hypothesis. Then the partition obtained by switching $b_1$ and $c_1$ yields a better partition where the ordering of the $b_i$ and $c_i$ are chosen to satisfy \eqref{assum}.}
\label{mod}
\end{figure}
Namely to get from $\cP_*$ to $\cP'$ we keep the configuration unchanged in the dyadic cubes $v_i$ for all $i\ge 3.$ Now in $\cP_*\mid_{v_1}$ and $\cP_*\mid_{v_2}$ there are $a_1$ and $a_2$ points respectively. 
We also keep the environments $\cP_*\mid_{w_i}$ and $\cP_*\mid_{z_i}$ unchanged for $i\ge 2$. However the only change we make is to interchange the environments $\cP_*\mid_{w_1}$ and $\cP_*\mid_{z_1}.$

In other words $\cP'\mid_{z_1}=\cP_*\mid_{w_1}$ and $\cP'\mid_{w_1}=\cP_*\mid_{z_1}.$
 We now have the following lemma
 
 \begin{lem}\label{better}Under the above assumptions 
 $H(\cP')< H(\cP_*).$
 \end{lem}
Before proving the above, let us see how this implies Theorem \ref{gsc}. As mentioned above, the lemma implies that $|a_i-a_j|\le 1$ for all $1\le i, j\le 2^d.$ Induction then takes care of the structures inside $v_i$ for $1\le i \le 2^d$.
\qed

\begin{proof}[Proof of Lemma \ref{better}]The intuition is that $\cP'$ is  more uniformly spread than $\cP_*$ and hence the energy should decrease. To formalize the above it would be convenient to separate the contributions from pairs of vertices depending on which dyadic cube they lie in. We need to develop some notation first.  For any $v\in \cD,$ and a partition $\cP$, let $H(P\mid_{v})$ denote the total energy from the interactions inside $v$ i.e., $$H(P\mid_{v})=\sum_{x_{s}\neq x_t \in \cP\mid_{v}} w(x_s,x_t).$$ Similarly for $u\neq v$ let $H(P\mid_{u}, P\mid_{v})$ denote the energy from the interactions across the vertices $u$ and $v$
i.e., $$H(P\mid_{u}, P\mid_{v})=\sum_{x_{s} \in \cP\mid_{u}, x_t \in \cP\mid_{v} } w(x_s,x_t).$$ Given these notations, for brevity, 
 $1\le i\neq j \le 2^d$ let
\begin{align*}
E_i:= H(\cP_{*}\mid_{v_i}), \quad\quad & E'_i:= H(\cP'\mid_{v_i}), \\
E_{i,j}:= H(\cP_{*}\mid_{v_i},\cP_{*}\mid_{v_j}), \quad\quad & E'_{i,j}:= H(\cP' \mid_{v_i},\cP_{*}\mid_{v_j}).
\end{align*}
Thus $E_i, E'_i, E_{i,j} E'_{i,j}$ denote the interactions in $v_i,$ and between $v_i$ and $v_j$ in the partitions $\cP_*$ and $\cP'$ respectively.
\begin{align}\label{rep1}
H(\cP_{*})&=\sum_{i=1}^{2^d} E_i + \sum_{1\le i\neq j \le 2^d} E_{i,j} \text{ and similarly,\,\,} 
H(\cP')=\sum_{i=1}^{2^d} E'_i +\sum_{1\le i\neq j \le 2^d} E'_{i,j}.
\end{align}

Now by construction of $\cP'$ from $\cP_*$, the following lemma follows.
\begin{lem}\label{cons12}
The following holds:
\begin{enumerate}
\item $E_i\neq E'_i$ for $i \le 2.$
\item $E_{i}=E'_{i}$ for all $i \ge 3.$
\item $E_{i,j}=E'_{i,j}$ if $\min (i,j)\ge 3$.
\item  Even though for any $j\ge 3$ we have $E_{1,j} \neq E'_{1,j}$ as well as $E_{2,j} \neq E'_{2,j}$ it follows that $E_{1,j}+E_{2,j}=E'_{1,j}+E'_{2,j}.$
\end{enumerate}
\end{lem}
\begin{proof}The first three observations are trivial. 
The last observation follows by noticing that for any $j \neq 1,2,$ given the number of points in $v_j,$ the quantity, 
$E_{1,j}+E_{2,j}$ only depends on the total number of points in $v_1$ and $v_2.$ Now our method of obtaining $\cP'$ from $\cP_*$ clearly conserves the last quantity and hence the observation.

\end{proof}

This along with \eqref{rep1} suggests that it suffices to show that 
$$E_1+E_2+2E_{1,2}> E'_1+E'_2+2E'_{1,2}.$$
We will now further decompose $E_1, E_1'$ into contributions from the various $w_i's$ and similarly $E_2, E_2'$ into contributions from $z_i's$ respectively. Thus for $i,j=1,2,\ldots, 2^d,$ let us define,
\begin{align*}
F_i:= H(\cP_{*}\mid_{w_i}), \quad\quad & F'_i:=H(\cP' \mid_{w_i}), \\
G_i:= H(\cP_{*}\mid_{z_i}), \quad\quad  & G'_i:= H(\cP'\mid_{z_i}), \\
F_{i,j}:= H(\cP_{*}\mid_{w_i},\cP_{*}\mid_{w_j}), \quad\quad & F'_{i,j}:= H(\cP'\mid_{w_i},\cP'\mid_{w_j}),\\
G_{i,j}:= H(\cP_{*}\mid_{z_i},\cP_{*}\mid_{z_j}), \quad\quad & G'_{i,j}:= H(\cP'\mid_{z_i},\cP'\mid_{z_j}),\\
H_{i,j}:= H(\cP_{*}\mid_{z_i},\cP_{*}\mid_{w_j}), \quad\quad & H'_{i,j}:= H(\cP'\mid_{z_i},\cP'\mid_{w_j}),\\
\end{align*}
Thus $F_i, G_i, F_{i,j}, G_{i,j}, H_{i,j}$ and  $F'_i,G'_i, F'_{i,j}, G'_{i,j}, H'_{i,j}$ denote the interactions in $w_i, z_i$ between the pairs $(w_i, w_j), (z_i, z_j)$ and $(z_i, w_j)$ in the partitions $\cP_*$ and $\cP'$ respectively.
With the above notation we have 
\begin{align}\label{rep2}
E_1=\sum_{i=1}^{2^d} F_i + \sum_{1\le i\neq j \le 2^d} F_{i,j} ,  \quad\quad &
E'_1=\sum_{i=1}^{2^d} F'_i + \sum_{1\le i\neq j \le 2^d} F'_{i,j} , \\
\nonumber
E_2=\sum_{i=1}^{2^d} G_i + \sum_{1\le i\neq j \le 2^d} G_{i,j} ,  \quad\quad &
E'_2=\sum_{i=1}^{2^d} G'_i + \sum_{1\le i\neq j \le 2^d} G'_{i,j} , \\
\nonumber
E_{1,2}= \sum_{1\le i,j \le 2^d} H_{i,j},  \quad\quad & E'_{1,2}= \sum_{1\le i,j \le 2^d} H'_{i,j}. 
\end{align}

We again have a  list of observations as in Lemma \ref{cons12}.

\begin{lem}\label{cons13}
The following holds:
\begin{enumerate}
\item $F_{i}=F'_{i}$ for all $i \ge 2,$  and $F_{1}=G'_{1}$ and $G_1=F'_1.$
and  $G_{i}=G'_{i}$ for all $i \ge 2.$
\item $F_{i,j}=F'_{i,j}$ if $\min (i,j)\ge 2$ and  $G_{i,j}=G'_{i,j}$ if $\min (i,j)\ge 2$.
\item $F_{1,j}=2^{d-2}H'_{j,1}$ and  $G_{1,j}=2^{d-2}H'_{1,j}$
\item $F'_{1,j}=2^{d-2}H_{j,1}$ and  $G'_{1,j}=2^{d-2}H_{1,j}$
\item $H_{i,j}=H'_{i,j}$ if $\min (i,j)\ge 2.$ and $H_{1,1}=H'_{1,1}$
\end{enumerate}
\end{lem}
\begin{proof} (1), (2), (5) are trivial.  (3) and (4) follow since for two points $x,y$ in $w_1,w_2$  $w(x,y)=2^{d-2}$ while for two points $x,y$ in $w_1,z_1,$  $w(x,y)=1.$
\end{proof}
Recall that we have to show $E_1+E_2+2E_{1,2}> E'_1+E'_2+2E'_{1,2}.$
From \eqref{rep2} and Lemma \ref{cons13} it follows that 
\begin{align*}
E_1-E'_1 &= F_1-F'_1+ 2\sum_{j\neq 1} (F_{1,j}-F'_{1,j}) \\
E_2-E'_2 &= G_1-G'_1+ 2\sum_{j\neq 1} (G_{1,j}-G'_{1,j})\\
E_{1,2}-E'_{1,2} &= \sum_{j}(H_{1,j}+H_{j,1})-(H'_{1,j}+H'_{j,1})
\end{align*}
Thus using the above it follows that we have to show 
\begin{align*}
\sum_{j\neq 1} (F_{1,j}-F'_{1,j})+\sum_{j\neq 1} (G_{1,j}-G'_{1,j})&> \sum_{j}(H'_{1,j}+H'_{j,1})-(H_{1,j}+H_{j,1}) \text{ or }\\
\sum_{j\neq 1} (F_{1,j}-F'_{1,j})+\sum_{j\neq 1} (G_{1,j}-G'_{1,j})&> \sum_{j\neq 1}(H'_{1,j}+H'_{j,1})-(H_{1,j}+H_{j,1}) \text{ or }\\
2^{d-2}\left[\sum_{j\neq 1} (H'_{j,1}-H_{j,1})+\sum_{j\neq 1} (H'_{1,j}-H_{1,j})\right]&> \sum_{j\neq 1}(H'_{1,j}+H'_{j,1})-(H_{1,j}+H_{j,1}) \text{ or }\\
\sum_{j\neq 1} \left[2^{d-2} (H'_{j,1}+H'_{1,j})+(H_{1,j}+H_{j,1})\right] & > \sum_{j\neq 1} \left[2^{d-2}(H_{j,1}+H_{1,j})+(H'_{1,j}+H'_{j,1})\right]\text{ or }\\
2^{d-2}\sum_{j\neq 1} (H'_{j,1}+H'_{1,j}) &> 2^{d-2}\sum_{j\neq 1}  (H_{j,1}+H_{1,j}).
\end{align*}

Now recalling \eqref{part98}, observe that for $j\neq 1,$ up to a deterministic multiplicative constant \begin{align}
H_{1,j}=b_1c_j, H'_{1,j}=c_1c_j,  H_{j,1}= c_1b_j, H'_{j,1}=b_1b_j
\end{align}
Thus we have to show 
\begin{align}
\sum_{j\neq 1} (b_1b_j +c_1c_j) >\sum_{j\neq 1 } (c_1b_j +b_1c_j) \text{ or }\\
\sum_{j\neq 1} (b_1-c_1)b_j >\sum_{j\neq 1 } (b_1-c_1)c_j 
\end{align}

Now by  \eqref{assum}, $ (b_1-c_1)\ge 1$ and $\sum_{j\ge 2}b_j >\sum_{j\ge 2}c_j$ and hence we are done.
\end{proof}

Given our understanding of the ground states from the above discussion,     we now proceed to finding an expression for the minimal energy stated in Theorem \ref{gse}.
\subsection{Proof of Theorem \ref{gse}} We now find $L_n$ by a recursion. In fact it will actually be useful to define the recursion in terms of the difference term $D_n:= L_{n+1}-L_n,$ for $n=0,1,2,\ldots$, where we set $L_0=L_1:=0$. (Hence $D_0=0$.)
\begin{lem}\label{rec1} $D_n$ satisfies the following recursion: 
\begin{align*}
D_n:= 2^{d-2}D_{\left\lfloor \frac{n}{2^d}\right\rfloor}+2(n-\left\lfloor \frac{n}{2^d}\right\rfloor).
\end{align*}
\end{lem}

\begin{proof}
The proof is a straightforward consequence of Theorem \ref{gsc}. To see this, let $\cP_1$ and $\cP_2$ be two ground state partitions for $n$ and $n+1$ respectively given by Theorem \ref{gsc}.  Now by definition of these partitions (see \eqref{induc0}), it follows that there exists $1\le i\le 2^d$, which without loss of generality, we can assume to be $1,$ such that:
\begin{align*}
\cP_2\mid_{v_i} =\cP_1\mid_{v_i}, \text {for } 2\le i \le 2^d,\,\,
\cP_2(v_1)  = \cP_1(v_1)+1\text{ and moreover},\,\,
\cP_1(v_1)&= \left\lfloor \frac{n}{2^d}\right\rfloor.
\end{align*}
Now, noting that the energy obtained due to interaction between points both of which are outside $v_1$ is the same in $\cP_1$ and $\cP_2$, one has  $H(\cP_2)-H(\cP_1)=A_1+A_2$, where $A_1$ is the difference in $\cP_1$ and $\cP_2$ of the contribution of pairs of points both of which lie in $v_1$.  Similarly $A_2$  is the difference in $\cP_1$ and $\cP_2$ of the  contribution of pairs of points, only one of which is in $v_1.$ Now by the definition of the hierarchical potential, the latter is only the interaction of the extra point in $\cP_2\mid_{v_1}$ not in $\cP_1\mid_{v_1}$, with all the points in $\cP_2\mid_{v_j}$ for $j\neq 1.$ Since $\cP_1$ and $\cP_2$ agree outside $v_1,$ it follows that the number of such points is $$n- \cP_1(v_1)=n-\left\lfloor \frac{n}{2^d}\right\rfloor.$$
Let $x_{n+1}\in v$ be the newly added point and $y$ be any of the $n-\lfloor \frac{n}{2^d}\rfloor$ points mentioned above. 
Then $w(x_{n+1},y)=1$ by definition of the potential. 
Thus $$A_2=2(n-\left\lfloor \frac{n}{2^d}\right\rfloor)$$ (the extra factor $2$ is because the energy functional $H_n$ counts $w(x_{n+1},y)$ as well as $w(y,x_{n+1})$).
We now claim $A_1=2^{d-2}D_{\lfloor \frac{n}{2^d}\rfloor}.$ To see this, note that by the recursive nature of the ground state partition, the partitions $\cP_1$ and $\cP_2$ attain the ground states with ${\lfloor \frac{n}{2^d}\rfloor}$ and ${\lfloor \frac{n}{2^d}\rfloor}+1$ particles respectively. Moreover since $v$ is a cube of side length $1/2,$ due to the nature of the potential 
\begin{align*}
H(\cP_{1}\mid_{v})=2^{d-2}L_{{\left\lfloor \frac{n}{2^d}\right\rfloor}} \text{ and similarly }
H(\cP_{2}\mid_{v})=2^{d-2}L_{{\left\lfloor \frac{n}{2^d}\right\rfloor}+1}.
\end{align*}
Thus $A_1=2^{d-2}\left(L_{{\left\lfloor \frac{n}{2^d}\right\rfloor}+1}-L_{{\left\lfloor \frac{n}{2^d}\right\rfloor}}\right)=2^{d-2}D_{\left\lfloor \frac{n}{2^d}\right\rfloor}.$ Putting the above together the lemma follows.
\end{proof}

Henceforth, it will be slightly convenient to work with the sequence $E_n=D_n-2n$ for $n=0,1,2,\ldots$. Using Lemma \ref{rec1}, it follows that, for $n\geq 2$,
\begin{equation}\label{rec2}
E_{n}=2^{d-2}E_{\left\lfloor \frac{n}{2^d}\right\rfloor}+(2^{d-1}-2)\left\lfloor \frac{n}{2^d}\right\rfloor.
\end{equation}

To solve the above recursion it will be cleaner to write $n$ in base $2^d.$ Let $n=\sum_{i=0}^kc_i2^{di}$ where  $0\le c_i\le2^d-1$ for all $i\le k$ and $c_k>0.$
We now have the following lemma.
\begin{lem}\label{l:dn} For $d\geq 3$ and $n\geq 1$,
$$
E_n=C_d\left[n -\sum_{i=0}^k c_i 2^{{(d-2)}i}\right]$$ and hence 
$D_n=(C_d+2)n -C_d\sum_{i=0}^k c_i 2^{(d-2)i}$, 
where 
$C_d=\frac{2^{d-1}-2}{3\cdot2^{d-2}},$ was defined in Theorem \ref{gse}.
\end{lem}

\begin{proof}
The statement about $E_n$ follows by  applying recursion \eqref{rec2} repeatedly and using the boundary condition that $E_0=0$. To see this, observe that $\left\lfloor \frac{n}{2^d}\right\rfloor=\sum_{i=0}^{k-1}c_{i+1}2^{di}$ and hence from the recursion,
\begin{eqnarray*}
E_n&=&(2^{d-1}-2)\left[\sum_{i=0}^{k-1}c_{i+1}2^{di}+2^{d-2}\sum_{i=0}^{k-2}c_{i+2}2^{di}+2^{2(d-2)}\sum_{i=0}^{k-3}c_{i+3}2^{di}+\ldots\right]\\
&=&(2^{d-1}-2)\left[\sum_{i=1}^kc_i2^{(d-2)(i-1)}(1+4+4^2+\ldots+4^{i-1})\right]\\
&=&\frac{2^{d-1}-2}{3}\left[\sum_{i=1}^kc_i2^{(d-2)(i-1)}(4^i-1)\right]\\
&=&C_d\left[\sum_{i=0}^k c_i2^{di}-\sum_{i=0}^k c_i2^{(d-2)i}\right]\\
&=&C_d\left[n -\sum_{i=0}^k c_i 2^{{(d-2)}i}\right]\,.
\end{eqnarray*}

The second statement follows by using the  relation between $D_n$ and $E_n.$
\end{proof} 

We can now finish the proof of Theorem \ref{gse}.
\begin{proof}[Proof of Theorem \ref{gse}] This follows directly from the expression of $D_n$ in Lemma \ref{l:dn} by computing $\sum_{i=1}^{n-1}D_i$ to get $L_n.$ 
\end{proof}

\begin{rmk} It is instructive to note that in the special cases when $n=c_k2^{dk}$ for some $1\leq c_k<2^d,$ one obtains $$L_{n}=\frac{ (C_d+2) n^2}{2}- C_dn 2^{(d-2)k}\left(\frac{(2^d-1)2^{d-1}}{2^{d-2}-1} + \frac{(c_{k}-1)}{2}\right)+O(n)\,,$$
where $C_d$ is as in Theorem \ref{gse}. Now by choosing different values of $1\le c_k\le 2^d-1$ and $k$ going to infinity, the above sequence of examples show that  $$\frac{L_{n}-\frac{(C_d+2)n^2}{2}}{n^{\frac{2d-2}{d}}}=\Theta(1)$$ but does not have a limit as $n\to \infty$.  
\end{rmk}

Before jumping in to the proof of Theorems \ref{marigid}, \ref{keyrigid}, as a warm up, we see how the above estimates already allow us to obtain a sharper estimate of $Z(n,\beta)$ than the one appearing in \cite{cha}.
\begin{lem}\label{linearcorrection}There exists an universal constant $C$ such that for all large enough $n$ and any $\beta$ the following holds,
$$-\beta L_n -Cn\le \log Z(n,\beta)\le -\beta L_n$$
\end{lem} 
Before proving the above, we record a lemma which follows immediately from the proof of Theorem \ref{gsc}.

\begin{lem}Given $n$, let $k$ be such that $2^{d(k-1
)}<n\le 2^{dk}.$ Then there exists a configuration $\cP_{\rm min}$ as in \eqref{induc0} where each element of $\cD_k$ contains at most one point.
\end{lem}
Note that as mentioned before $\cP_{\rm min}$ is not necessarily unique. 
\begin{proof} 
We will compare the optimal configurations $\cP_1$ and $\cP_2$ described recursively by \eqref{induc1} for $n$ and $m=2^{dk}.$
As in \eqref{induc1}, for $v_1,v_2,\ldots, v_{2^d}$ denoting the members of $\cD_1$ arranged in some order, 
\begin{equation}\label{induc1}
\cP_{1}(v_i)=\left \{ \begin{array}{cc}
\lfloor \frac{n}{2^d}\rfloor+1& \,1\le i \le r,\\
\lfloor \frac{n}{2^d}\rfloor &   r< i.
\end{array}
\right.
\end{equation}
where $n=2^d\lfloor \frac{n}{2^d}\rfloor+r.$ Now similarly note that for all $v_i$,
$\cP_{2}(v_i)=2^{d(k-1)}$. Thus applying the recursive structure of the definitions of $\cP_1$ and $\cP_2$ and simple induction, it follows that for all $v\in \cD_k$ $$\cP_1(v)\le \cP_2(v)=1$$ and hence we are done.
\end{proof}
We now provide the proof of Lemma \ref{linearcorrection}.
\begin{proof} 
Since $\log Z(n,\beta)\le -\beta L_n$ is trivially true; we will focus on the other inequality. 
Let $k$ be as in the previous lemma and consider $\cP_2$. Let $\cD_{k,{\rm occ}} \subset \cD_k$ be the set of all cubes $v$ such that $\cP_2(v)=1$
Since the particles are exchangeable, there are $n! $ possible ways to arrange the above. Moreover the volume of every sub-cube $v$ is $\frac{1}{2^{dk}}\ge \frac{1}{2^d n}.$
Thus using \eqref{partition} it follows that 
$
Z_n \ge e^{-\beta L_n} \frac{n!}{(2^dn)^n}
\ge e^{-\beta L_n -(\log 2^d +1)n}.
$
\end{proof} 

We are now ready to proceed with the key section in this paper, proving Theorem \ref{keyrigid}.

\section{Strong hyperuniformity bounds for dyadic cubes} \label{pot1}
This section contains many of the new ideas in the paper. However it is quite long and technical and hence for the reader's benefit we start with the following roadmap describing what the various subsections achieve. 
\begin{enumerate}
\item In Subsection \ref{s1} we record some useful definitions and lemmas to be used throughout the rest of the section. 
\item Subsection \ref{s2} contains some of the key new techniques in this paper which are used to perform a refined Peierls argument type Energy vs Entropy comparison  as outlined in Section \ref{iop} to establish formal versions of \eqref{prelim1} and \eqref{ratioestimate} stated earlier.  These are then used to prove \eqref{key20} which estimates contributions from configurations which are not ground states beyond level $m_0=\Theta(\log\log n)$. The key result in this section is Proposition \ref{l:entropy1}.

\item In Subsection \ref{s3}, using the above results, we obtain a precise estimate of the partition function comparing it with the contribution of the ground states as was stated in \eqref{keyest1}. 

\item In Subsection \ref{s4}, putting everything together we finish the proof of Theorem \ref{keyrigid}. 
\end{enumerate}
We start with some more notations.  
We will let $\Omega=\Omega_n$ be the set of all partitions of $n$ into the different sub-cubes of the tree $\cD.$ Furthermore, for any set $A$ of configurations, we define
\begin{equation}\label{partcont}
Z(A):=\int_{\mathbf{x}\in A} e^{-\beta H(\mathbf{x})}d\mathbf{x}\,.
\end{equation} 
We shall call the various hierarchical steps as \textit{levels}, for $i=0,1,2,\ldots$, the $i$-th level has sub-cubes of side length $2^{-i}$. In the tree notation, all vertices that are at a graph distance of $i$ from the root are said to be at level $i$, for $i=0,1,2,\ldots$.  For any dyadic sub-cube $v$, $ \mathrm{Lev}(v)$ will denote the level of the cube. 
\subsection{Some useful lemmas}\label{s1} In this subsection, we record and prove some lemmas that will be needed in the following subsections. Skipping the proofs in this subsection at first read will not affect the ability to perceive the logic of the arguments in the subsequent sections. For convenience, we introduce another notation; consider the base $2^d$ representation of $n$, that is
$n=\sum_{i=0}^kc_i2^{id}\,,$
where $0\leq c_i\leq 2^d-1$ and $c_k>0$. Define
\begin{equation}\label{e:defg}
\gamma(n):=\sum_{i=0}^kc_i2^{i(d-2)}\le 4n^{\frac{d-2}{d}}\,,
\end{equation}
where the last inequality follows by using that the function $x\to x^{\frac{d-2}{d}}$ is concave and hence sub-additive.
We next show that the function $\gamma(\cdot)$ is sub-additive.

\begin{lem}\label{l:gn}For $d\geq 3$, $n,r\in \N$,
\[\gamma(n+r)\leq \gamma(n)+\gamma(r)\,.\]
\end{lem}
\begin{proof}First we prove that for any $m$, with $m=\sum_{i=0}^s a_i2^{id}$, 
\begin{equation}\label{e:add1}
\gamma(m+1)\leq \gamma(m)+1\,.
\end{equation}
To see this, observe that, if the smallest $i\in \{0,1,2,\ldots,s\}$ such that $a_i<2^d-1$ is  $t$, then
\[\gamma(m+1)=\sum_{i=t+1}^s a_i2^{i(d-2)}+(a_t+1)2^{(d-2)t}\,.\]
Now, if $t=0$, then clearly
$\gamma(m+1)=\gamma(m)+1\,.$
If $t\geq 1$, then for $d\geq 3$,
$$(2^d-1)\sum_{i=0}^{t-1}2^{i(d-2)}=\frac{(2^d-1)(2^{t(d-2)}-1)}{2^{d-2}-1}\geq 2^{(d-2)t}\,.$$
Thus,
\begin{align*}
\gamma(m+1)&=\sum_{i=t+1}^s a_i2^{i(d-2)}+(a_t+1)2^{(d-2)t}=\sum_{i=t}^s a_i2^{i(d-2)}+2^{(d-2)t}\\
&\leq \sum_{i=t}^s a_i2^{i(d-2)}+ (2^d-1)\sum_{i=0}^{t-1}2^{i(d-2)}=\sum_{i=0}^s a_i2^{i(d-2)}=\gamma(m)\,.
\end{align*}

Now, consider the base $2^d$ representations of $r$ as
$r=\sum_{i=0}^\l b_i2^{id}\,,$ and let 
$r_i=b_i2^{id} \mbox{ for } i=0,1,\ldots,\l\,.$
Clearly $r=\sum_{i=0}^\l r_i$  and 
\begin{equation}\label{e:gr}
\gamma(r)=\sum_{i=0}^\l \gamma(r_i)\,.
\end{equation}
Now, for any $m$, we show that for all $i=0,1,\ldots,\l$,
\begin{equation}\label{e:oneatatime}
\gamma(m+r_i)\leq \gamma(m)+\gamma(r_i)\,.
\end{equation}
Let $m=\sum_{i=0}^s a_i2^{id}$. If $a_i+b_i\leq 2^d-1$, then $\gamma(m+r_i)=\gamma(m)+\gamma(r_i)\,.$
On the other hand, if $a_i+b_i=2^d+q$ with $0\leq q\leq 2^d-2$, define $m'_i=\sum_{j=0}^{s-i-1} a_{i+1+j}2^{jd}$, and observe that
\begin{align*}
\gamma(m+r_i)&=\gamma(m_i'+1)2^{(i+1)(d-2)}+q2^{i(d-2)}+ \sum_{j=0}^{i-1}a_j2^{j(d-2)},\\
&\leq  (\gamma(m_i')+1)2^{(i+1)(d-2)}+q2^{i(d-2)}+ \sum_{j=0}^{i-1}a_j2^{j(d-2)},\\
&= \gamma(m_i')2^{(i+1)(d-2)}+(2^{d-2}+q)2^{i(d-2)}+ \sum_{j=0}^{i-1}a_j2^{j(d-2)},\\
&< \gamma(m_i')2^{(i+1)(d-2)}+(a_i+b_i)2^{i(d-2)}+ \sum_{j=0}^{i-1}a_j2^{j(d-2)}=\gamma(m)+\gamma(r_i)\,,
\end{align*}
where the first inequality follows from \eqref{e:add1} and the second inequality follows because $a_i+b_i=2^d+q>2^{d-2}+q$.
Finally, by a repeated application of \eqref{e:oneatatime}
\begin{align*}
\gamma(n+r)=\gamma(n+r_0+r_1+\ldots+r_\l)
\leq \gamma(n+\sum_{i=0}^{\l-1}r_i)+\gamma(r_\l)\leq \ldots
\leq  \gamma(n)+\sum_{i=0}^\l \gamma(r_i)=\gamma(n)+\gamma(r)
\end{align*}
where the last equality follows from \eqref{e:gr}.
\end{proof}
Using the above we have, 
\begin{lem}\label{l:d-2byd}For all $n,r\in \N$,
\[\gamma(n+r)\leq \gamma(n)+4r^{\frac{d-2}{d}}\,.\]
\end{lem}
\begin{proof}By Lemma \ref{l:gn} and \eqref{e:defg},
$\gamma(n+r)-\gamma(n)\leq \gamma(r)\leq 4r^{\frac{d-2}{d}}\,.$
\end{proof}

To prove results of the form \eqref{key20}, we need estimates on the number of different ground states.   This is formally defined next. 
Let 
\begin{equation}\label{baselevel}
h_n:=\left\lceil\frac{\log n}{\log 2^d}\right\rceil\,.
\end{equation}
Note that by Theorem \ref{gsc}, for any ground state, each cube at level $h_n$ has at most one particle. We will call $h_n$ as the \textit{base level}. If we enumerate the sub-cubes at the base level (there are $2^{dh_n}$ of them), then the number of points in each sub-cube (which is either $0$ or $1$) determines the ground state configuration.
\begin{defn}
Let $b_n$ denote the number of ground state configurations with $n$ points. 
\end{defn}
Note that the positions of the points in the base sub-cubes do not matter. Also, here we do not enumerate the points, so that the permutation of the points going to the various sub-cubes is not considered. Thus, $b_n$ denotes the number of ways of choosing the $n$ dyadic sub-cubes at the base level that are occupied in a ground state configuration.
Then for any ground state, $2^{-ndh_n}$ is the product of the volumes of the $n$ occupied sub-cubes at the base level, each of which contains one of the $n$ points. Armed with these notations, we state the following technical lemma comparing $b_n, h_n$ and $b_{n+1}, h_{n+1}$ which will be used in the later sections.
\begin{lem}\label{l:bn}There exists a positive constant $c$ such that for all $n\in\N\cup \{0\}$,
 \[e^{-c\log (n+2)}\leq \frac{b_{n+1}2^{-(n+1)dh_{n+1}}}{b_n2^{-ndh_n}}\leq e^{c\log(n+2)}\,.\] 
As an immediate consequence we get for any $k\in \Z,$ with $n+k\ge 0,$ for a slightly large constant $c,$
\begin{equation}\label{iter123}
 \frac{b_{n+k}2^{-(n+k)dh_{n+k}}}{b_n2^{-ndh_n}}\leq e^{c|k|\log (n+2)}\,\, \text{for }k< 0,\,\, \frac{b_{n+k}2^{-(n+k)dh_{n+k}}}{b_n2^{-ndh_n}}\leq e^{c|k|\log (n+k+2)}\,\,\,\text{for }k\ge 0\,.
\end{equation} 
\end{lem}
\begin{proof} We first prove the first statement. Note that $h_{n+1}-h_n\in\{0,1\}.$ Thus we break the proof into two cases.\\

\noindent
$\bullet$ $h_{n+1}=h_n:$  Given a ground state with $n$ particles, we bound the number of ways to place the $n+1^{th}$ point to obtain a ground state with $n+1$ particles. At every level, there are at most $2^d$ sub-cubes where the extra particle can be placed. Since $h(n)\leq c\log n$ for some constant $c>0$, we have the upper bound.  (In fact, one gets one as an upper bound in this case.) For the lower bound, we start with a ground state of size $n+1$ and remove a point to obtain one of size $n.$  We proceed by choosing one of the $2^{h_nd}$ sub-cubes in $\cD_{h_{n+1}}$ and removing the point in it, if it has one. So \[\frac{b_{n}}{b_{n+1}}\leq 2^{h_nd}=2^d2^{(h_n-1)d}\leq 2^dn\,,\]
since $2^{d(h_n-1)}\leq n\leq 2^{dh_n} $. This gives the lower bound.

\noindent
$\bullet$ $h_{n+1}=h_n+1:$ In this case, $n=2^{h_nd}$, $b_n=1$, $b_{n+1}=2^{-1}n(2^d-1)(2^d)^{n}$.  To see $b_{n+1}$, observe that, there are $n$ choices for the sub-cube where the extra point is placed, and in the next level, each point that singly occupies a sub-cube choses one of the $2^d$ many children to move to, and the only sub-cube that has two points in it has ${{2^d}\choose{2}}$ choices for the two children that will each contain one point. Thus 
\[\frac{b_{n+1}2^{-(n+1)dh_{n+1}}}{b_n2^{-ndh_n}}=\frac{2^d-1}{2^{d+1}}\,,\]
which satisfies both the inequalities. 

The second statement now follows by iterating the first statement $|k|$ times.
\end{proof}

 \subsection{Key estimates: Entropy vs. energy comparisons}\label{s2} 
 The main goal of this subsection is to show that configurations that differ from ground state configurations at certain levels away from the root, do not contribute much to the partition function; in the process making precise the steps indicated in \eqref{prelim1} and \eqref{ratioestimate}. We start with a few more notations. Recall that for a partition $\cP$ and any sub-cube $v$, $\cP(v)$ denotes the number of points in $v$ under the partition $\cP$.
 For two partitions $\cP_1$ and $\cP_2$ that specify the number of points landing in the different dyadic cubes, if
 \[\cP_1(v)=\cP_2(v)\]
 for all vertices $v$ with $\mathrm{Lev}(v)\leq r$, they will be said to agree till level $r$. 
 If for any partition $\cP$, and any dyadic cube $v$, 
 \[\cP\mid_{v}=\cP^{(\cP(v))}_{\min}\,,\]
 where $\cP_{\min}^{(m)}$ is some ground state partition as defined in \eqref{induc0} and $\cP\mid_{v}$ is the partition $\cP$ restricted to the subtree rooted at $v$, we say that the subtree rooted at $v$ has a ground state configuration. If
 \[\cP\mid_{v}=\cP^{(\cP(v))}_{\min}\,,\]
 for all $v$ with $\mathrm{Lev}(v)=r$, we say that the partition has a ground state configuration at levels beyond $r$ (we will need to define further notation describing such situations, for the sake of exposition later in the article). 
 Also, for any finite sequence/vector $\bx=(x_1,x_2,\ldots,x_t)$, with $x_i\in \Z_{\geq 0}$, define the set of vectors obtained by permuting the coordinates; i.e.,
\[[\bx]= \left\{(x_{\sigma(1)},x_{\sigma(2)},\ldots,x_{\sigma(t)}): (\sigma(1),\ldots,\sigma(t)) \text{ is a permuation of } (1,2,\ldots,t)\right\} \,.\]
For two vectors $\bx_1$ and $\bx_2$ of the same size, define
\begin{equation}\label{e:dist}
\mathrm{dist}(\bx_1,\bx_2)=\sqrt{\min\{d_2^2(z_1,z_2): z_1\in [\bx_1], z_2\in [\bx_2]\}}\,,
\end{equation}
where $d_2(\cdot,\cdot)$ is the $L_2$ distance between two vectors. By a standard rearrangement inequality, the minimum is attained in the above expression when the vectors $\mathbf{x}_1$ and $\mathbf{x}_2$ are ordered similarly (say both in decreasing order). Also, given any vector $\bk=(k_1,k_2,\ldots,k_t)$ of the same size as that of $\bx$, we say
\begin{equation}\label{e:defsum}
[\bx]+\bk=[\by]\,,
\end{equation}
for some $\by=(y_1,\ldots,y_t)$ if
\[y_{\sigma'(i)}=x_{\sigma(i)}+k_i, \ \ \text{ for } i=1,2,\ldots,t\,,\]
where $\sigma$ and $\sigma'$ are two permutations of $(1,2,\ldots,t)$ such that $y_{\sigma'(1)}\geq y_{\sigma'(2)}\ldots\geq y_{\sigma'(t)}$ and $x_{\sigma(1)}\geq x_{\sigma(2)}\geq \ldots \geq x_{\sigma(n)}$. Clearly, then $\mathrm{dist}(\bx,\by)=\|\bk\|_2=\sqrt{\sum_{i=1}^t k_i^2}$.

For the following lemma, consider two partitions $\cP_1$ and $\cP_2$ that agree till level $m-1$. Let $\cP_1=(\mathbf{a}_1,\mathbf{a}_2,\ldots,\mathbf{a}_{2^{d(m-1)}})$, where $\mathbf{a}_i=(a_{i,1},\ldots,a_{i,2^d})$ denotes the number of points in the $2^d$ descendants of the $i$th sub-cube (from left to right in the natural planar embedding of the $2^{d}-$ary tree $\cD$) at level $m-1$. And, suppose $\cP_2=(\mathbf{e}_1,\mathbf{e}_2,\ldots,\mathbf{e}_{2^{d(m-1)}})$ has a ground state configuration at levels beyond $m-1$, where  $\mathbf{e}_i=(e_{i,1},\ldots,e_{i,2^d})$ denotes the number of points in the $2^d$ descendants of the $i$th sub-cube at level $m-1$.

Recalling from Section \ref{ground10}, the notation $H(\cP)$ for the value of the Hamiltonian for the partition $\cP$, the next lemma gives a quantitative lower bound on $H(\cP_1)-H(\cP_2)$.
\begin{lem}\label{l:energy}
There exists a positive constant $c'$ such that for all $m\in \N$, all $d\geq 3$, and all partitions $\cP_1,\cP_2$ defined as above,
\[H(\cP_1)\geq H(\cP_2)+c'2^{(d-2)m}\sum_i\mathrm{dist}^2(\mathbf{a}_i,\mathbf{e}_i)\,.\]
\end{lem}
Observe that, it is important to have the $\mathrm{dist}(\cdot,\cdot)$ function defined in terms of the permutations of the vectors $\ba_i$ and $\mathbf{e}_i$ as in \eqref{e:dist}. For instance, if $\ba_i$ and $\mathbf{e}_i$ are permutations of each other for each $i$ and $\cP_1,\cP_2$ both have ground state configurations beyond level $m$, then they are two different ground state configurations beyond level $m-1$, and as a consequence $H(\cP_1)=H(\cP_2)$. Thus in this case $\mathrm{dist}(\ba_i,\mathbf{e}_i)$ ought to be zero for each $i.$
\begin{proof}Since we are proving a lower bound, it suffices to assume that $\cP_1$ has a ground state configuration beyond level $m.$ Furthermore, it is enough to restrict ourselves to the descendants of an individual sub-cube in $\cD_{m-1}$, since the contributions to $H(\cP_1)$ or $H(\cP_2)$ by the interactions between particles across sub-cubes  in $\cD_{m-1},$ are the same.  We start by considering the descendants of the first sub-cube only. Let us call the $2^d$ descendent sub-cubes of the first sub-cube of $(m-1)$-th level as $v_1,v_2,\ldots,v_{2^d}$. Also, if $(x^{(i)}_1,x^{(i)}_2,\ldots,x^{(i)}_n)$ denotes the points corresponding to the $i$-th partition for $i=1,2$, we define
\[H(\cP_i,v_j)=\sum_{s\neq u; x^{(i)}_s,x^{(i)}_u\in v_j} w(x^{(i)}_s,x^{(i)}_u)\,\ \ \ \text{ for } i=1,2;\ j=1,2,\ldots,2^d\,,\]
i.e., the energy coming from the interaction inside the sub-cube $v_j.$
Furthermore, if $v(x)$ denotes the vertex in $\cD_{m-1},$ that contains the point $x$, let
\[I(\cP_i)=\sum_{v(x^{(i)}_s)\neq v(x^{(i)}_u)} w(x^{(i)}_s,x^{(i)}_u)\,\ \ \ \text{ for } i=1,2\,,\] i.e., the energy coming from the interaction across sub-cubes $\{v_j\, : j=1,2, \ldots, 2^d\}.$
Before giving a full proof of Lemma \ref{l:energy},
as a warm up, we first compare the energies of some pairs of partitions that differ only in a small number of sub-cubes.\\

\noindent
$\bullet$  \textbf{Partitions differing only in two coordinates}:
 To begin with, consider two partitions $\mathbf{f}_1:=(b_1+1,b_2-1, b_3,\ldots,b_{2^d})$ and $\mathbf{f}_2:=(b_1,b_2,\ldots,b_{2^d})$ with $b_1\geq b_2$, where the vector entries denote the number of points in the vertices $(v_1,v_2,\ldots,v_{2^d})$. Recall that by hypothesis both the partitions have ground state configurations beyond level $m$. Then,
 \[H(\mathbf{f}_1)-H(\mathbf{f}_2)=\sum_{i=1,2}H(\mathbf{f}_1,v_i)-\sum_{i=1,2}H(\mathbf{f}_2,v_i)+I(\mathbf{f}_1)-I(\mathbf{f}_2)\,.\]
 By definition (see Section \ref{mdmr}), we have, 
 \[I(\mathbf{f}_1)-I(\mathbf{f}_2)=2^{(d-2)(m-1)}\Big((b_1+1)(b_2-1)-b_1b_2\Big)=-2^{(d-2)(m-1)}(b_1-b_2+1)\,.\]
 And since the partitions have ground state configurations beyond level $m$, recalling the notation $D_n=L_{n+1}-L_n$, we get,
 \[\sum_{i=1,2}\left(H(\mathbf{f}_1,v_i)-H(\mathbf{f}_2,v_i)\right)= 2^{(d-2)m}\sum_i(L_{b_1+1}+L_{b_2-1}-L_{b_1}-L_{b_2})=2^{(d-2)m}(D_{b_1}-D_{b_2-1})\,.\]
 Also, recall from \eqref{e:defg}, that $\gamma(m):=\sum_{i=0}^kc_i2^{(d-2)i}$ where $m=\sum_{i=0}^kc_i2^{di}$. Clearly
$\gamma(m)\leq m\,.$
 Then by Lemma \ref{l:dn},
 \begin{align*}
 D_{b_1}-D_{b_2-1}&=(C_d+2)(b_1-b_2+1)-C_d(\gamma(b_1)-\gamma(b_2-1))\\
 &=C_d[b_1-b_2+1-(\gamma(b_1)-\gamma(b_2-1))]+2(b_1-b_2+1)\,.
 \end{align*}
 Now, by Lemma \ref{l:gn}, we have
$\gamma(b_1)-\gamma(b_2-1)\leq \gamma(b_1-b_2+1)\leq b_1-b_2+1\,.$
 Thus
\begin{align*}C_d[b_1-b_2+1-&(\gamma(b_1)-\gamma(b_2-1))]\geq 0\,\, \text{ and hence,}\\
D_{b_1}-D_{b_2-1}&\geq 2(b_1-b_2+1)\,.
\end{align*}
 Putting all this together, we have,
 \[H(\mathbf{f}_1)-H(\mathbf{f}_2)\geq 2^{(d-2)m+1}(b_1-b_2+1)-2^{(d-2)(m-1)}(b_1-b_2+1)\geq 2^{(d-2)m}(b_1-b_2+1)\,.\]
 
 Next, we consider two partitions $\mathbf{f}_1:=(b_1+k,b_2-k, b_3,\ldots,b_{2^d})$ and $\mathbf{f}_2:=(b_1,b_2,\ldots,b_{2^d})$ with $b_1\geq b_2$ and $k\in \N$. We can move from $\mathbf{f}_2$ to $\mathbf{f}_1$ in $k$ steps as follows;
 \[(b_1,b_2,\ldots,b_{2^d})\rightarrow (b_1+1,b_2-1,\ldots,b_{2^d})\rightarrow\ldots(b_1+k-1,b_2-(k-1),\ldots,b_{2^d})\rightarrow (b_1+k,b_2-k,\ldots,b_{2^d})\,,\]
 where we move one point each time, such that, at all every step, the first two vector entries are arranged in a non-increasing order, and the corresponding partition has a ground state configuration beyond level $m$. Thus, we can apply the previous argument $k$ times to get
\[H(\mathbf{f}_1)-H(\mathbf{f}_2)\geq 2^{(d-2)m}\sum_{i=0}^{k-1}(b_1-b_2+1+2i)\geq 2^{(d-2)m}k^2\,,\]
since $b_1\geq b_2$.\\

\noindent
$\bullet$  \textbf{Partitions with only the first coordinate increasing}:
Next, we consider the situation where $\mathbf{f}_1:=(b_1+k_1,b_2-k_2, b_3-k_3,\ldots,b_s-k_s,b_{s+1},\ldots,b_{2^d})$ and $\mathbf{f}_2:=(b_1,b_2,\ldots,b_{2^d})$ with $b_1\geq b_i$ for $i=2,3,\ldots,s$ and $k_1,\ldots,k_s \in \N, 2\leq s\leq 2^d$ such that $k_1=k_2+\ldots+k_s$, and the partitions have ground state configurations beyond level $m$. Then we move from $\mathbf{f}_2$ to $\mathbf{f}_1$ as 
\begin{align*}
(b_1,b_2,\ldots,b_{2^d})&\rightarrow (b_1+k_2,b_2-k_2,\ldots,b_{2^d})\rightarrow (b_1+k_2+k_3,b_2-k_2,b_3-k_3\ldots,b_{2^d})\rightarrow\\ 
&\ldots\rightarrow(b_1+k_1,b_2-k_2,\ldots,b_s-k_s,b_{s+1},\ldots,b_{2^d})\,.
\end{align*}
Applying the above arguments $s-1$ times, we have (we use the convention that $\sum_{i=a}^bc_i=0$ for $b<a$)
\begin{align}\label{e:good}
H(\mathbf{f}_1)-H(\mathbf{f}_2)&\geq 2^{(d-2)m}\sum_{j=2}^s\sum_{i=0}^{k_j-1}\left(b_1+\sum_{r=2}^{j-1}k_r-b_{j}+1+2i\right)\\
&\geq 2^{(d-2)m}\sum_{j=2}^s\sum_{i=0}^{k_j-1}\left(\sum_{r=2}^{j-1}k_r+1+i\right)\nonumber\\
&= 2^{(d-2)m}\sum_{i=1}^{k_2+\ldots+k_s} i\geq 2^{(d-2)m}\frac{k_1^2}{2}\nonumber\,.
\end{align}
Here the second inequality uses $b_1\geq b_i$ for $i=2,3,\ldots,s$.\\

\noindent
$\bullet$  \textbf{General case}: Combining the above observations, we now finish proof of the general case for two partitions $\cP_1$ and $\cP_2$. Recall that we have assumed that  $\cP_1$ has a ground state configuration beyond level $m$, as otherwise, $H(\cP_1)$ only increases. And as argued before, we just consider the partition restricted to the descendants of the first sub-cube in $\cD_m$. Let $a_1\geq a_2\geq \ldots\geq a_{2^d}$ and $e_1\geq \ldots \geq e_{2^d}$ be the numbers in $\mathbf{a}_1=(a_{1,1},\ldots,a_{1,2^d})$ and $\mathbf{e}_1=(e_{1,1},\ldots,e_{1,2^d})$, arranged in decreasing order. Let
 \[a_i=e_i+k_i, \ \ \ \text{ where } k_i\in \Z, \ i=1,2,\ldots,2^d\,.\]
 {Then, as mentioned already, a simple rearrangement inequality implies, from the definition in \eqref{e:dist}
 \[\mathrm{dist}^2(\mathbf{a}_1,\mathbf{e}_1)=\sum_{i=1}^{2^d} k_i^2\,.\]}
Now recall that $\mathbf{e}$ partitions the particles according to a ground state given the configuration at level $m-1.$ Thus by Theorem \ref{gsc},  $$e_i-e_j\in \{0,1\}$$ for any $1\leq i\leq j\leq 2^d$. Now this along with the facts that 
\begin{align*}
a_1\geq a_2 \ldots\geq a_{2^d};&\,\,\,
 e_1\geq e_2 \ldots \geq e_{2^d}; \text{ and }\\
 \sum_i a_i&=\sum_i e_i
\end{align*}

 imply that all the positive $k_i's$  appear before the negative $k_i's$.  {Let $k_i\geq 0$ for $i=1,2,\ldots,s$ and $k_i\le 0$ for $i>s$ for some $s\in \{1,2,\ldots,2^d\}$}. Then, using the arguments in \eqref{e:good} repeatedly, we have
\[H(\mathbf{a}_1)-H(\mathbf{e}_1)\geq c2^{(d-2)m}\sum_{i=1}^{2^{d}} k_i^2\mathbf{1}(k_i>0)=c2^{(d-2)m}\sum_{i=1}^{s} k_i^2\geq c'2^{(d-2)m}\sum_{i=1}^{2^d} k_i^2\,,\]
where $c'=\frac{c2^{-d}}{1+2^{-d}}$. For the last inequality, we have used the fact that
\[k_1^2+\ldots+k_s^2\geq s^{-1}(k_1+\ldots+k_s)^2\geq 2^{-d}(k_{s+1}+\ldots+k_{2^d})^2\geq 2^{-d}(k_{s+1}^2+\ldots+k_{2^d}^2)\,,\]
since $k_1+\ldots+k_s=k_{s+1}+\ldots+k_{2^d}$. 
This proves the lemma.
\end{proof}
The next result Proposition \ref{l:entropy1} is the key result of this section, making precise \eqref{key20}. 
 However we need to develop some more notations first. 
\subsubsection{More Notations:}
Let $G=G_n$ denote the set of all ground state configurations. 
It would also be convenient to define $G^{(r)}:=G_n^{(r)}$ to be the set of all partitions that are ground states beyond level $r$ i.e., the partitions restricted to any dyadic box of size $2^{-r}$ is a ground state partition. 
Thus, given any vector $\ba$ of size $2^{dm}$ that assigns the number of particles in the various sub-cubes  in $\cD_m$ (from left to right in the natural planar embedding of $\cD$), we will be identifying it with the subset of partitions in $G^{(m)}$ whose partition at level $m$ is equal to $\ba.$ We will call this set as $G(a):=G_n(a)=G_n^{(m)}(a).$  
Let 
\begin{equation}\label{not32}
\bs(\ba)=(\bs_1,\bs_2,\ldots,\bs_{2^{dm}})
\end{equation}
 denote a ground state partition at level $m+1$, where $\bs_i$ is a vector of length $2^d,$ denoting the partition restricted to the descendants of the $i$-th sub-cube in $\cD_m$ (arranged from left to right). Furthermore, to avoid ambiguity  we assume that for each $i$, the entries of the vector $\bs_i$ are arranged in decreasing order.
Now, given a vector $\ba$ as above, we would need to consider the set of all configurations 
\begin{equation}\label{ground13}
G^{(m+1)}(\ba):=G_{n}^{(m+1)}(\ba)\subset G_{n}^{(m+1)},
\end{equation}
 which induce partition $\ba$ on $\cD_m$, i.e., all partitions which agree with $\ba$ on $\cD_m$ and are ground states restricted to elements of $\cD_{m+1}$ but might not be elements of $G(\ba).$
 \begin{figure}[t]
\centering
\includegraphics[width=.5\textwidth]{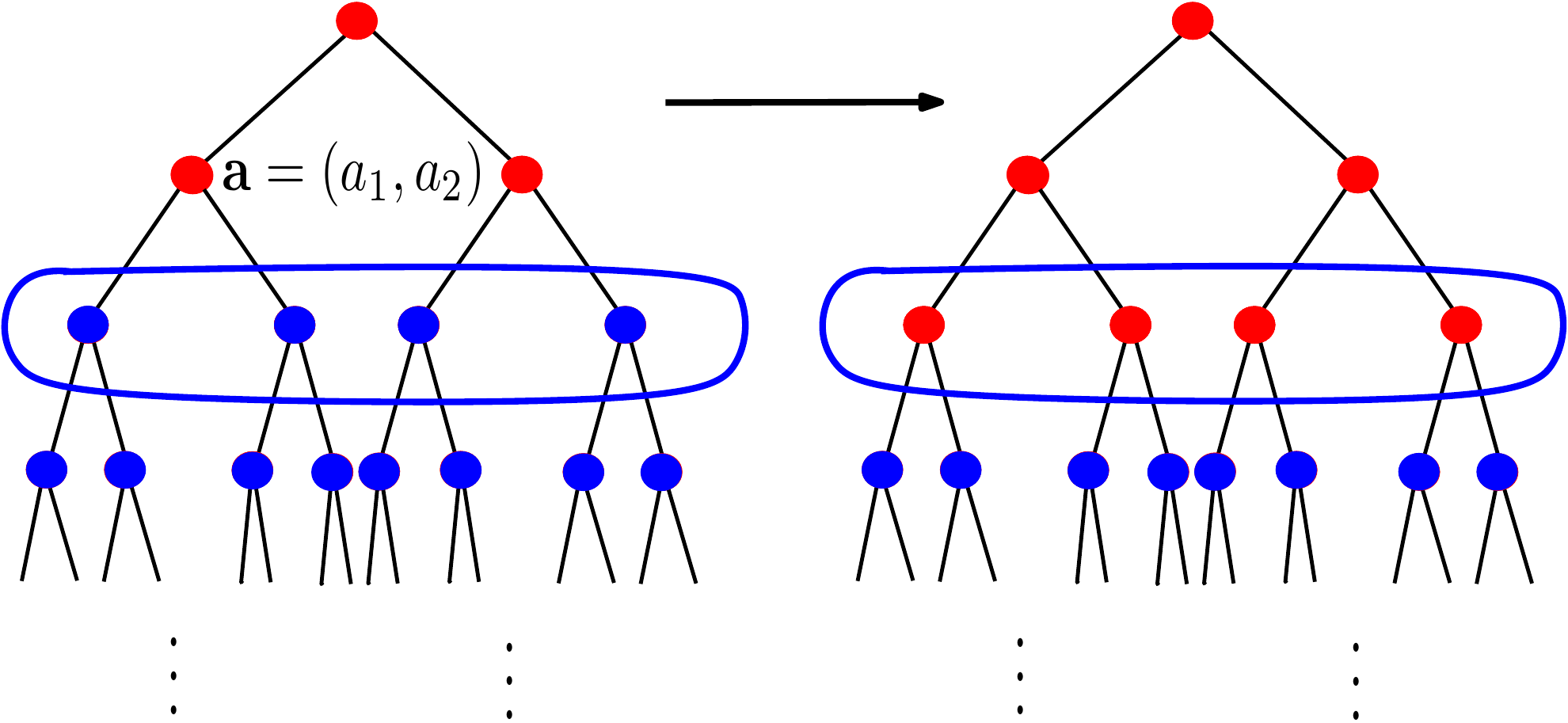}
\caption{A toy example of a situation where Lemma \ref{l:energy} lower bounds the energy difference between the two configurations shown in the figure where $(a_1,a_2)$ denotes the partition at the second level, but the partitions differ at the third level and as in Figure \ref{fig1}, blue vertices denote ground states. }
\label{fig3}
\end{figure}

To do this the following notation would be quite convenient. 
Consider a collection of vectors $$\bg_1:=(\bk_1,\bk_2,\ldots,\bk_{2^{dm}}).$$ Here each $\bk_i$ will be a vector of length $2^d$ with entries in $\Z$ which add to $0$. and is the vector that gives the differences in the numbers of particles in various descendants of the $i$th sub-cube in $\cD_m$ from the vector $\bs_i.$  
Formally by \eqref{e:defsum} we identify $G^{(m+1)}(\ba)$ with the union of the following vectors of size $2^{d(m+1)}$:  
\begin{equation}\label{ident13}
\bigcup_{\bg_1}\left([\bs_1]+\bk_1\right)\times \left([\bs_2]+\bk_2\right)\ldots \times \left([\bs_{2^{dm}}]+\bk_{2^{dm}}\right), \end{equation} where the union is over all vectors $\bg_1$ for which the above objects make sense (i..e $\bg_1$ cannot have a huge negative entry etc.) Thus the above gives us a way to identify subsets of $G^{(m+1)}(a)$ with the vectors $\bg_1.$
Note that according to this notation, $G(\ba)=G^{(m)}(\ba)$ can be identified with $\bg_1=\mathbf{0}.$
However for our purposes, we will need to consider slightly more complicated version of the above.  Similar to \eqref{ground13}, we define for any $r>0,$
\begin{equation}\label{ground14}
G^{(m+r)}(\ba):=G_{n}^{(m+r)}(\ba)\subset G_{n}^{(m+r)}
\end{equation} as the set of all partitions on $\cD_{m+r}$ which induce partition $\ba$ on $\cD_m.$ Thus similar to \eqref{ident13} we will identify the same with $r$ tuples of vectors $(\bg_1,\bg_2,\ldots, \bg_r).$ To this end we define for any $m,r\in \N$,  the set
\[\Gamma^{(r)}_m=\left\{(\bg_1,\bg_2,\ldots,\bg_r)\right\}\,,\]
where the vectors $\bg_r$ are of the following form.
\[\bg_r=(\bk^{(r)}_1,\bk^{(r)}_2,\ldots,\bk^{(r)}_{2^{d(m+r-1)}})\,,\]
such that for each $i=1,2,\ldots,2^{d(m+r-1)}$, $\bk^{(r)}_i=(k^{(r)}_{i,1},k^{(r)}_{i,2}\ldots,k^{(r)}_{i,2^d})$ with $k^{(r)}_{i,j}\in \Z$ for $j=1,2,\ldots,2^d$ and
$\sum_{j=1}^{2^d}k^{(r)}_{i,j}=0\,.$  We will also consider the natural embedding of $\Gamma_m^{(r)} \subset \Gamma_m^{(r+1)}$ where we choose $\bg_{r+1}$ to be $(\mathbf{0}, \mathbf{0},\ldots, \mathbf{0}).$
Finally let, 
 \[\Gamma_m=\bigcup_{j\ge 0}\Gamma^{(j)}_m.\]

Now by the above correspondence, for any such partition $\mathbf{a}$ of length $2^{dm}$ and any vector $\nu=(\bg_1,\ldots,\bg_r)\in \Gamma^{(r)}_m$, we denote by $\nu(\ba)$, the subset of partitions in $G^{(m+r)}(\ba)$ corresponding to $\nu,$ i.e., where the difference vectors at subsequent levels up to permutations are given by the vectors $\bg_1,\bg_2,\ldots, \bg_r.$ The formal definition is by induction:
\begin{itemize}
\item  If $r=0$, $\nu(\ba)=\ba$ (so for $r=0$ we get $G^{(m)}(a)$). 
\item 
Now, letting $\nu'=(\bg_1,\ldots,\bg_{r-1})\in \Gamma^{(r-1)}_{m}$, and by induction having defined $\nu'(\ba)\subset G^{(m+r-1)}(a)$ we now proceed to defining $\nu(a).$ For any $\sigma\in \nu'(\ba)$, let $\bs(\sigma)=(\bs_1',\bs_2',\ldots,\bs'_{2^{d(m+r-1)}})$ as defined in \eqref{not32}.  For each $i=1,2,\ldots, 2^{d(m+r-1)}$, consider the equivalence class
\begin{equation}\label{e:dessi}
[\bs_i]=[\bs_i']+\bk^{(r)}_i\,,
\end{equation}
where per our convention $\bs_i$ is the unique vector in $[\bs_i']+\bk_i$ which is sorted in decreasing order,
 $\bg_r=(\bk^{(r)}_1,\ldots,\bk^{(r)}_{2^{d(m+r-1)}})$ and the operation on equivalence classes was defined in \eqref{e:defsum} (provided the definition makes sense, that is $(\bs_1,\bs_2,\ldots,\bs_{2^{d(m+r-1)}})$ is a valid partition of $n$ points). Given the above let 
\begin{equation}\label{e:Asig}
A_\sigma=[\bs_1]\times[\bs_2]\times\ldots \times [\bs_{2^{d(m+r-1)}}]
\end{equation}
We finally define
\[\nu(\ba)=\bigcup_{\sigma\in \nu'(\ba)}A_\sigma\,\]
\end{itemize}
to be the union of the sets of configurations corresponding to different elements of $\nu'(\ba).$ 
Thus
\begin{equation}\label{impdef12}
G^{(m+r)}(\ba)=\bigcup_{\nu \in  \Gamma^{(r)}_{m}}\nu(\ba)
\end{equation} and hence we will identify $G^{(m+r)}(\ba)$ and  $\Gamma^{(r)}_{m}.$
Note that $G^{(m+r)}\subset G^{(m+r+1)}$ since  $\Gamma^{(r)}_{m}\subset \Gamma^{(r+1)}_{m}.$
We now state the following key result which roughly says that the set of all configurations, that are not ground state configurations beyond level $\Theta(\log\log n)$, does not contribute much to the partition function. Recall that $\Omega$ denotes the set of all partitions. 
\begin{ppn}\label{l:entropy1}Let $m_0=m_0(\beta)$ be the smallest integer such that \[c'\beta2^{(d-2)m_0}\geq \max\{4c\log n,8\log(2n+1)\}\,,\] where $c$ is as in Lemma \ref{l:bn} and $c'$ is as in Lemma \ref{l:energy}. Then 
for all large enough $n,$
\[Z(\Omega\setminus G^{(m_0)})\leq C_0n^{-c_0}Z_n\,,\]
for some constants $c_0,C_0>0$ and $Z_n$ was defined in \eqref{partition}.
\end{ppn}

Recalling the notations $b_n$ and $h_n$ from Lemma \ref{l:bn} we first record a simple but useful observation.
\begin{lem}\label{ldefZP} Given $m,$ and a vector $\ba=(n_1,n_2,\ldots,n_{2^{dm}})$, with $\sum_{i=1}^{2^{dm}} n_i=n$,
\begin{equation}\label{e:defZP}
Z(G(\ba)):=\int_{\bx\in G(\ba)}e^{-\beta H(\bx)}d\bx=e^{-\beta H(G(\ba))}\int_{\bx\in \cP}d\bx=e^{-\beta H(G(\ba))}n!\prod_{i=1}^{2^{dm}}b_{n_i}2^{-n_id(m+h_{n_i})}\,,
\end{equation}
where $H(G(\ba))$ is the common value of the Hamiltonian for all configurations in $G(\ba).$
\end{lem}
\begin{proof}
The first equality follows by the fact which has been observed before that   the value of the Hamiltonian only depends on the partition and not the precise location of the points.  
To see the last equality, note that there are $n!$ ways to arrange the particles once we have chosen the $n$ sub-cubes where they are placed. 
By definition of $b_m$, for any $m,$ the number of ground state configurations (that determine the sub-cubes at the base level that are occupied) is $\prod_{i=1}^{2^{dm}}b_{n_i}.$ Furthermore, by the definition of $h_m,$ the description of a ground state partition below the $i$th vertex in $\cD_m,$  will go down to  $\cD_{m+h_{n_i}}.$  Each such cube  has volume $2^{-d(m+h_{n_i})}$ and $n_i$ such cubes will be occupied as each has at most one particle (see Figure \ref{fig4}).
\end{proof}
\begin{figure}[t]
\centering
\includegraphics[width=.5\textwidth]{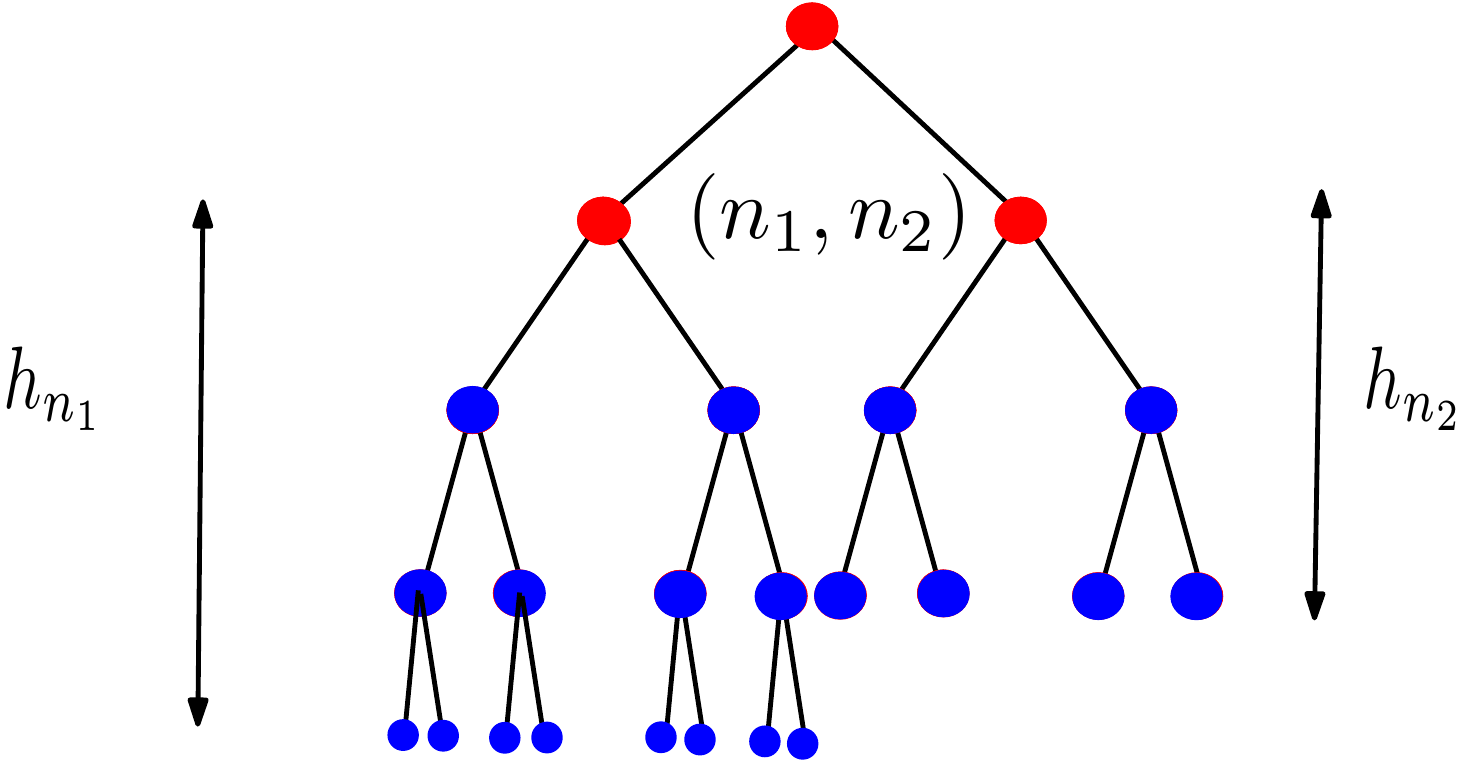}
\caption{Illustration describing the setting of Lemma \ref{ldefZP}.}
\label{fig4}
\end{figure}

The following lemma is now immediate. 
\begin{lem}\label{groundcontribution} Recalling that $G_n$ denotes the set of all ground state configurations with $n$ points, 
\[Z(G_n)=e^{-\beta L_n}n!b_n2^{-ndh_n}\,.\]
\end{lem}
Given the above preparation we now embark on proving the above proposition. The argument will involve some computations but the key idea will be an energy-entropy comparison outlined in Section \ref{iop}. 
\begin{proof}[Proof of Proposition \ref{l:entropy1}]
Let $\nu=(\bg_1,\ldots,\bg_r) \in \Gamma^{(r)}_{m_0}$, where  $\bg_r=(\bk_1,\ldots,\bk_{2^{d(m_0+r-1)}})$ with $\bk_i=(k_{i,1},\ldots,k_{i,2^d})$ (we are suppressing the $r$ dependence on the $\bk_i's$ for brevity). Let $\nu'=(\bg_1,\ldots,\bg_{r-1})$. 
Also fix a partition $\ba$ of length $2^{dm_0}$ and any $\sigma\in \nu'(\ba)$. We shall compare $Z(A_\sigma)$ with that of $Z(\sigma)$ (recall the definition of $A_\sigma$ from \eqref{e:Asig}).  Recalling \eqref{not32}, let $\bs(\sigma)=(\bs_1',\ldots,\bs_{2^{d(m_0+r-1)}}')$ with $\bs_i'=(s'_{i,1},\ldots,s'_{i,2^d})$. By \eqref{e:Asig}, $\sigma$ and $A_\sigma$ are identified with subsets of $G^{((m_0+r))}(\ba)$ given by $$[\bs'_1]\times[\bs'_2]\times\ldots \times [\bs'_{2^{d(m+r-1)}}], \text{ and }[\bs_1]\times[\bs_2]\times\ldots \times [\bs_{2^{d(m+r-1)}}]$$ respectively where $[\bs_i]=[\bs'_i]+k_i$. 
Note that, for all configurations $\tau \in A_\sigma$, $H(\tau)$ is a constant, and thus we shall use $H(A_\sigma)$ to denote the common value.  By Lemma \ref{l:energy} and the definition of $\mathrm{dist}(\cdot,\cdot)$, we have
\begin{equation}\label{e:compenergy}
H(A_\sigma)\geq H(\sigma)+c'2^{(d-2)(m_0+r)}\sum_{i=1}^{2^{d(m_0+r-1)}}\sum_{j=1}^{2^{d}}k_{i,j}^2\,.
\end{equation}

Next, define 
\[\mathscr{S}(\sigma)=\prod_{i=1}^{2^{d(m_0+r-1)}}\big|[\bs_i']\big|\,, \ \ \ \ \mathscr{S}(A_\sigma)=\prod_{i=1}^{2^{d(m_0+r-1)}}\big|[\bs_i]\big|\,,\]
where $|\cdot|$ denotes the cardinality of a set. Then, using \eqref{e:defZP}, we get
\[Z(\sigma)=e^{-\beta H(\sigma)}n!\mathscr{S}(\sigma)\prod_{i,j}b_{s'_{i,j}}2^{-s'_{i,j}d\left(m_0+r+h_{s'_{i,j}}\right)}\,,\]
\[Z(A_\sigma)=e^{-\beta H(A_\sigma)}n!\mathscr{S}(A_\sigma)\prod_{i,j}b_{s_{i,j}}2^{-s_{i,j}d\left(m_0+r+h_{s_{i,j}}\right)}\,.\]
Note that the proposition asks us to bound the ratio $\frac{Z(\Omega\setminus G^{(m_0)})}{Z_n}.$ We will start by bounding $\frac{Z(A_{\sigma})}{Z(\sigma)}.$

If $\bk_i=\mathbf{0}$ (the vector $(\underset{2^{d} \text{ times}}{\underbrace{0,0,\ldots 0}})$) for some $i$, then $[\bs_i]=[\bs_i']$ and that factor cancels in the ratio $\mathscr{S}(A_\sigma)/\mathscr{S}(\sigma)$. For all other $i$, since $[\bs_i]/[\bs_i']\leq 2^d!\leq e^{2^d\log (2^d)}$, one has
\[\frac{\mathscr{S}(A_\sigma)}{\mathscr{S}(\sigma)}\leq e^{2^d\log(2^d)\sum_i \mathbf{1}(\bk_i\neq \mathbf{0})}\,.\]
Using this along with the expressions of $Z(\sigma)$ and $Z(A_\sigma)$, we get
\[\frac{Z(A_\sigma)}{Z(\sigma)}\leq\frac{e^{-\beta H(A_\sigma)}e^{2^d\log(2^d)\sum_i \mathbf{1}(\bk_i\neq \mathbf{0})}\prod_{i,j} b_{s_{i,j}}2^{-s_{i,j}d\left(m_0+r+h_{s_{i,j}}\right)}}{e^{-\beta H(\sigma)}\prod_{i,j} b_{s'_{i,j}}2^{-s'_{i,j}d\left(m_0+r+h_{s'_{i,j}}\right)}}\,.\]
From here, using \eqref{e:compenergy}, we have
\begin{align}\nonumber
\frac{Z(A_\sigma)}{Z(\sigma)}&\leq\frac{e^{-c'\beta 2^{(d-2)(m_0+r)}\sum_{i,j} k_{i,j}^2}e^{2^d\log(2^d)\sum_i \mathbf{1}(\bk_i\neq \mathbf{0})}\prod_{i,j} b_{s_{i,j}}2^{-s_{i,j}d\left(m_0+r+h_{s_{i,j}}\right)}}{\prod_{i,j} b_{s'_{i,j}}2^{-s'_{i,j}d\left(m_0+r+h_{s'_{i,j}}\right)}}.\end{align}
Now our choice of $m_0$ implies that for $r\geq 0$ (and for all large enough $n$),
\[\frac{c'\beta}{2}2^{(d-2)(m_0+r)}\sum_{i,j} k_{i,j}^2\geq 2^d\log(2^d)\sum_{i,j} k_{i,j}^2\geq 2^d\log(2^d)\sum_i \mathbf{1}(\bk_i\neq \mathbf{0})\,.\]
Using this the above simplifies to 
\begin{equation}\label{e:imp}
\frac{e^{-c_1\beta 2^{(d-2)(m_0+r)}\sum_{i,j} k_{i,j}^2}\prod_{i,j} b_{s_{i,j}}2^{-s_{i,j}d\left(m_0+r+h_{s_{i,j}}\right)}}{\prod_{i,j} b_{s'_{i,j}}2^{-s'_{i,j}d\left(m_0+r+h_{s'_{i,j}}\right)}},\,
\end{equation}
where $c_1=\frac{c'}{2}.$
Since $\sum_{i,j}s_{i,j}=\sum_{i,j}s'_{i,j}=n$,  \eqref{e:imp} is the same as 
\[\frac{e^{-c_1\beta 2^{(d-2)(m_0+r)}\sum_{i,j} k_{i,j}^2}\prod_{i,j} b_{s_{i,j}}2^{-s_{i,j}dh_{s_{i,j}}}}{\prod_{i,j} b_{s'_{i,j}}2^{-s'_{i,j}dh_{s'_{i,j}}}}\,.\]
Finally, since $[\bs_i]=[\bs'_i]+\bk_i$, we will use  the second statement  of Lemma \ref{l:bn} to compare the factors $b_{s'_{i,j}}2^{-s'_{i,j}dh_{s'_{i,j}}}$ and $b_{s_{i,j}}2^{-s_{i,j}dh_{s_{i,j}}}$ appearing in the products above (we will use the fact that $s'_{i,j},s_{i,j}\le n$). Plugging all of these we get,
\begin{align}\label{e:imp2'}
\frac{Z(A_\sigma)}{Z(\sigma)}&\leq\frac{e^{-c_1\beta 2^{(d-2)(m_0+r)}\sum_{i,j} k_{i,j}^2}\prod_{i,j} b_{s_{i,j}}2^{-s_{i,j}dh_{s_{i,j}}}}{\prod_{i,j} b_{s'_{i,j}}2^{-s'_{i,j}dh_{s'_{i,j}}}} 
\nonumber\\
&\overset{\eqref{iter123}}{\leq}\frac{e^{-c_1\beta 2^{(d-2)(m_0+r)}\sum_{i,j} k_{i,j}^2}\prod_{i,j} e^{c|k_{i,j}|\log n}b_{s'_{i,j}}2^{-s'_{i,j}dh_{s'_{i,j}}}}{\prod_{i,j} b_{s'_{i,j}}2^{-s'_{i,j}dh_{s'_{i,j}}}},  
\nonumber
\\
\nonumber
&= e^{-c_1\beta 2^{(d-2)(m_0+r)}\sum_{i,j} k_{i,j}^2+\sum_{i,j}c|k_{i,j}|\log n}\\
&\leq e^{-c''\beta 2^{(d-2)(m_0+r)}\sum k_{i,j}^2}\,,
\end{align}
where $c''=\frac{c_1}{2}$. Here, the last inequality follows because for $r\geq 0$,
\begin{equation}\label{ass1}
\frac{c_1\beta}{2}2^{(d-2)(m_0+r)}\sum_{i,j} k_{i,j}^2\geq c\log n \sum_{i,j} k_{i,j}^2 \geq c\log n\sum_{i,j} |k_{i,j}| \,,
\end{equation}
 (this is where we crucially use our choice of $m_0$).
Since $\nu(\ba)=\bigcup_{\sigma\in \nu'(\ba)}A_\sigma$ and the bound in \eqref{e:imp2'} is independent of our choice of $\sigma\in \nu'(\ba),$ we finally have from \eqref{e:imp2'},
\begin{equation}\label{ratio98}
\frac{Z(\nu(\ba))}{Z(\nu'(\ba))}\leq \sup_{\sigma\in \nu'(\ba)}\frac{Z(A_\sigma)}{Z(\sigma)}\leq e^{-c''\beta 2^{(d-2)(m_0+r)}\sum_{i,j} k_{i,j}^2}\,.
\end{equation}
Now recalling the identification between $G^{(m_0+r)}(\ba)$ and $\Gamma^{(r)}_m$ we see that 
 $$
 \frac{Z(G^{(m_0+r)}(\ba))}{Z(G^{(m_0+r-1)}(\ba))}=\frac{\displaystyle{\sum_{\nu'\in \Gamma^{(r-1)}_{m_0}}\sum_{\bg_r}}Z(\nu(\ba))}{\displaystyle{\sum_{\nu'\in \Gamma^{(r-1)}_{m_0}}}Z(\nu'(\ba))}
 $$
where $\nu'=(\bg_1,\ldots,\bg_{r-1})$ and $\nu=(\nu',\bg_r).$ Thus using \eqref{ratio98} we get that the RHS above is bounded by 
\begin{equation}\label{ratio99}
\sum_{\bg_r}\exp\left\{-c''\beta 2^{(d-2)(m_0+r)}\sum_{i=1}^{2^{d(m_0+r-1)}}\sum_{j=1}^{2^{d}} [k^{(r)}_{i,j}]^2\right\},
\end{equation}
where $\bg_r=(\bk^{(r)}_1,\bk^{(r)}_2,\ldots, \bk^{(r)}_{2^{d(m_0+r-1)}}).$  
Hence, defining $G^{(\infty)}(\ba)=\bigcup_{r=1}^{\infty}G^{(m_0+r)}(\ba)$ we get
\begin{eqnarray}\label{e:imp2}
\frac{Z(G^{(\infty)}(\ba)\setminus G^{(m_0)}(\ba))}{Z(G^{(m_0)}(\ba))}&\leq& \sum_{\begin{subarray}{c}\mathbf{g}_{r},r=1,2,\ldots,\\
 \text{ {\small at least one} } 
\mathbf{g}_{r}\neq 0 \end{subarray}
}\exp\left\{-c''\beta \sum_r \Big[2^{(d-2)(m_0+r)}\sum_{i,j} [k^{(r)}_{i,j}]^2\Big]\right\},
\nonumber\\
&\leq & \prod_{r=1}^\infty\Bigl(1+\sum_{\mathbf{g}_{r}\neq \mathbf{0}}\exp\left\{-c''\beta  2^{(d-2)(m_0+r)}\sum_{i,j} [k^{(r)}_{i,j}]^2\right\}\Bigr)-1\,.
\end{eqnarray}
Now, if 
\[q(\bg_r)=\sum_{i,j} \ind(k^{(r)}_{i,j}\neq 0)\,,\]
 then for any fixed $q\in \N$, there are at most ${{2^{d(m_0+r)}}\choose{q}}(2n+1)^q\leq 2^{d(m_0+r)q}(2n+1)^q$-many vectors $\bg_r$ such that $q(\bg_r)=q$. To see this, observe that each of the non-zero entries in $\bg_r$ can take values in $\{-n,-n+1,\ldots,n\}$. Thus, for any $r\in \N$, again using 
 \[\frac{c''\beta}{2}2^{(d-2)(m_0+r)}\geq \log (2n+1),
 \text{ and }\sum_{i,j}[ k_{i,j}{(r)}]^2\geq q \text{ for } q=q(\bg_r),\] we get, 

\begin{align*}
\sum_{\bg_r:q(\bg_r)=q\neq 0}\exp\left\{-c''\beta  2^{(d-2)(m_0+r)}\sum_{i,j} [k_{i,j}^{(r)}]^2\right\} &\leq 2^{d(m_0+r)q}(2n+1)^qe^{-c''q\beta 2^{(d-2)(m_0+r)}}\\
&\leq e^{-c_2q\beta 2^{(d-2)(m_0+r)}}\,,
\end{align*}
where $c_2=\frac{c''}{4}.$ Note the choice of $m_0$ is used in the last inequality. 
Thus,
\begin{align*}\sum_{\mathbf{g}_{r}\neq \mathbf{0}}\exp\left\{-c''\beta  2^{(d-2)(m_0+r)}\sum_{i,j} [k^{(r)}_{i,j}]^2\right\}\leq \sum_{q=1}^\infty e^{-c_2q\beta 2^{(d-2)(m_0+r)}}\leq Ce^{-c_2\beta 2^{(d-2)(m_0+r)}}\,,
\end{align*}
for some $C>0$.
Hence, from \eqref{e:imp2}, using the inequality $1+x\leq e^x$ and the fact that 
\[c_2\beta 2^{(d-2)m_0}=\frac{c'}{16}\beta 2^{(d-2)m_0}\geq \frac{c}{4}\log n\,,\]  we have
\begin{align*}
\frac{Z(G^{(\infty)}(\ba)\setminus G^{(m_0)}(\ba))}{Z(G^{(m_0)}(\ba))}&\leq \prod_{r=1}^\infty \Big(1+Ce^{-c_2\beta 2^{(d-2)(m_0+r)}}\Big)-1
\leq  \exp\Big(\sum_{r=m_0}^\infty Ce^{-c_2\beta 2^{(d-2)(r+1)}}\Big)-1\\
&\leq  \exp\Big(C'e^{-c_2\beta 2^{(d-2)(m_0+1)}}\Big)-1
\leq  \exp\Big(C'e^{-\frac{c}{4}\log n}\Big)-1\\
&= \exp\Big(C'n^{-\frac{c}{4}}\Big)-1
\leq  C_1C'n^{-\frac{c}{4}}\,,
\end{align*}
where the last inequality follows since $\frac{e^x-1}{x}\leq C_1$ for all $x\in (0,1]$ and we can choose $n$ large enough so that $C'n^{-\frac{c}{4}}\leq 1$. Now noting that $\displaystyle{\Omega=\bigcup_{\ba}G^{(\infty)}(\ba)}$ and $G^{(m_0)}=\displaystyle{\bigcup_{\ba}G^{(m_0)}(\ba)},$
summing up over all possible partitions $\ba$ at level $m_0$, we have the proposition with $C_0=C_1C'$ and $c_0=\frac{c}{4}$ . 
\end{proof}
 
Using the previous result we finally arrive at the crucial estimate on the partition function hinted at in Section \ref{iop}. Note the improvement over the warm-up result Lemma \ref{linearcorrection}.
\subsection{Sharp estimate of partition function}\label{s3}
\begin{ppn}\label{l:Zn}For $d\geq 3$,
\[e^{-\beta L_n}n!b_n2^{-ndh_n}\leq Z_n\leq e^{-\beta L_n+C_1(\beta)\log^6 n}n!b_n2^{-ndh_n}\,,\]
where $C_1(\beta)>0$ is a decreasing function of $\beta$.
\end{ppn}   
\begin{proof}The lower bounds is obvious by Lemma \ref{groundcontribution} since $Z_n\ge Z(G_n)$.
For the upper bound, we shall use the results from the previous section. By Proposition \ref{l:entropy1}, 
$Z(G^{(m_0)})\geq \frac{Z_n}{2}\,,$
for large $n$. Thus it is enough to bound $Z(G^{(m_0)})$. To this end,
we will compare $Z(G^{(m_0)})$ and $Z(G_n).$
Observe first that by our choice of $m_0$,  we have 
$2^{dm_0}\leq C\left(\frac{\log n}{\beta}\right)^{\frac{d}{d-2}}\,,$ 
for some universal constant $C>0$. 
Now by the same arguments as in \eqref{e:imp2'}, \eqref{ratio98} and \eqref{ratio99}
we get (by taking  $m=1$), 
\begin{align}\label{firstfew}
\frac{Z(G^{(m_0)})}{Z(G_n)}\le \prod_{r=1}^{m_0}\Bigl[\sum_{\mathbf{g}_{r}}\exp\left\{-c'\beta  2^{(d-2)r}\sum_{i,j} [k^{(r)}_{i,j}]^2+\sum_{i,j}c|k^{(r)}_{i,j}|\log n\right\}\Bigr]
\end{align}
Note that since $r\le m_0$ above we do not use the full conclusion of \eqref{e:imp2'} using \eqref{ass1} and keep the term $\sum_{i,j}c|k^{(r)}_{i,j}|\log n.$  
If $|k^{(r)}_{i,j}|\geq \frac{c}{c'\beta}\log n$ for some $i,j$, where $c,c'$ are as in the above display, then 
\begin{equation}\label{e:5.6.1}
\exp\left(-c'\beta2^{(d-2)r}[k^{(r)}_{i,j}]^2+c\log n|k^{(r)}_{i,j}|\right)\leq 1\,.
\end{equation}
And if $|k_{i,j}|\leq \frac{c}{c'\beta}\log n$, then  
\begin{equation}\label{e:5.6.2}
\left(-c'\beta2^{(d-2)r}[k^{(r)}_{i,j}]^2+c\log n|k^{(r)}_{i,j}|\right)\leq \left(\frac{c^2}{c'\beta}\log^2 n\right)\,.
\end{equation}
Since $r\le m_0$ we get $$-c'\beta2^{(d-2)r}\sum_{i,j}[k^{(r)}_{i,j}]^2+c\log n\sum_{i,j}|k^{(r)}_{i,j}|\le 2^{dm_0}\left(\frac{c^2}{c'\beta}\log^2 n\right).$$ 
Putting this together along with our choice of $m_0,$ we get that the exponential term in \eqref{firstfew}, is at most $\exp\left(C(\beta)(\log n)^{\frac{d}{d-2}+2}\right)$ where $C(\beta)=\max\left\{\frac{c^2C}{c'\beta^{\frac{d}{d-2}+1}},\frac{C\log (2^d)}{\beta^{\frac{d}{d-2}}}\right\}$ is a decreasing function of $\beta$.
Since, for a fixed $r\leq m_0$, the number of possible vectors $\bg_r$ of size $2^{dr}$ is at most
 \[(2n+1)^{2^{dr}}\leq (2n+1)^{2^{dm_0}}\leq e^{C'(\beta)(\log n)^{\frac{d}{d-2}+1}}\,,\] 
 where $C'(\beta)=2C\beta^{-\frac{d}{d-2}}$ is a decreasing function of $\beta$, the total number of possible vectors $\bg_r$ for all $r\leq m_0$ is at most $e^{m_0C'(\beta)(\log n)^{\frac{d}{d-2}+1}}$. Thus,
 we have from \eqref{firstfew}, as in \eqref{e:imp2},
 \begin{align*}
 \frac{Z(G^{(m_0)})}{Z(G_n)}&\leq\sum_{\bg_r,\, r=1,2,\ldots,m_0
 }e^{\sum_{r=1}^{m_0}C(\beta)(\log n)^{\frac{d}{d-2}+2}}\\
&\leq e^{m_0C'(\beta)(\log n)^{\frac{d}{d-2}+1}}\times e^{m_0C(\beta)(\log n)^{\frac{d}{d-2}+2}}\\
&\leq e^{C_0(\beta)\log( \log n) C(\beta)(\log n)^{\frac{d}{d-2}+2}}
\leq  e^{C_1(\beta)(\log n)^{\frac{d}{d-2}+3}}\leq e^{C_1(\beta)(\log n)^6}\,,
 \end{align*}
 for $d\geq 3$ and hence we are done. Here we have used that $m_0\leq C_0(\beta)\log \log n$, where $C_0(\beta)$ is a decreasing function of $\beta$. Thus, $C_1(\beta):=C_0(\beta)C(\beta)$ is decreasing in $\beta$.  \end{proof}

The above gives us a sharp estimate on the ratio of partition functions which as indicated in Section \ref{iop}, (see \eqref{cavity1}) will be used to prove Theorem \ref{keyrigid}.
\begin{lem}\label{l:Zratio}For $d\geq 3$ and any $n\geq 2$, there exists $C(\beta)>0$, a decreasing function of $\beta$, such that
\[e^{-\beta D_n-C(\beta)\beta\log^6 n}\leq \frac{Z(n+1,\beta)}{Z(n,\beta)}\leq e^{-\beta D_n+C(\beta)\beta\log^6 n}\,,\]
where we recall $D_n=L_{n+1}-L_n$.
\end{lem}
\begin{proof}Using Proposition \ref{l:Zn} and the upper bound in Lemma \ref{l:bn}, we have
\begin{align*}
\frac{Z(n+1,\beta)}{Z(n,\beta)}&\leq e^{-\beta D_n+C_1(\beta)\log^6(n+1)}(n+1)\frac{b_{n+1}2^{-(n+1)dh_{n+1}}}{b_n2^{-ndh_n}}
\leq e^{-\beta D_n+C_1(\beta)\log^6(n+1)}(n+1)e^{c\log (n+2)}\\
&\leq e^{-\beta D_n+C_1(\beta)\log^6(n+1)+c\log (n+2)+\log(n+1)}
\leq e^{-\beta D_n+C_2(\beta)\log^6n}
=e^{-\beta D_n+C(\beta)\beta\log^6n}\,,
\end{align*}
where $C_2(\beta)=2C_1(\beta)+2c+2$ is a decreasing function of $\beta$ where $C_1(\beta)$ is as defined in Lemma \ref{l:Zn}, and $C(\beta):=\frac{C_2(\beta)}{\beta}$ is also a decreasing function of $\beta$. Similarly, one has the lower bound using Proposition \ref{l:Zn} and the lower bound in Lemma \ref{l:bn}.
\end{proof}
Finally we are ready to prove Theorem  \ref{keyrigid}.
\subsection{Variance bounds: Proof of Theorem \ref{keyrigid}} \label{s4}  
Following the notations in  \cite[Lemma 2.5]{cha}, let
\begin{align*}
f(n_1,\ldots,n_{2^d}):=\frac{n!}{n_1!\ldots n_{2^d}!};\,\,
{g}(n_1,\ldots,n_{2^d}):=e^{-\beta\sum_{i\neq j}n_in_j};\,\,
h(n_1,\ldots,n_{2^d}):=\prod_{i=1}^{2^d}Z(n_i,2^{d-2}\beta)\,.
\end{align*}
Now, choose non negative integers $m_i$ such that $\sum_{i=1}^{2^d} m_i=n$ and $m_i\in \{\lfloor2^{-d}n\rfloor,\lfloor2^{-d}n\rfloor+1\}$. And take $k_i\in \Z$ such that $\sum_{i=1}^{2^d} k_i=0$ and $0\leq m_i+k_i\leq n$ for each $i$. 
Then,
\begin{align}\label{e:hratio}
&\frac{h(m_1+k_1,\ldots,m_{2^d}+k_{2^d})}{h(m_1,\ldots,m_{2^d})}
=\frac{\prod_{i=1}^{2^d}Z(m_i+k_i,2^{d-2}\beta)}{\prod_{i=1}^{2^d}Z(m_i,2^{d-2}\beta)}\\
\nonumber
&\leq \prod_{i=1}^{2^d}\exp\left(-2^{d-2}\beta(L_{m_i+k_i}-L_{m_i})+C(\beta)\beta|k_i|\log^6 n \right)\nonumber\\
\nonumber
&= \exp\left(-2^{d-2}\beta\sum_{i=1}^{2^d}(L_{m_i+k_i}-L_{m_i})+C(\beta)\beta \log^6 n\sum_{i=1}^{2^d}|k_i|\right)\,,
\end{align}
where the inequality follows from a repeated application of Lemma \ref{l:Zratio} and $C(\beta)$ is as defined in that lemma. 
Now from Theorem \ref{gse}, we have 
$$L_n=\frac{(C_d+2)n(n-1)}{2}-C_d\sum_{r=1}^{n-1}\gamma(r),$$
where $\gamma(m)$ was defined in \eqref{e:defg}. Thus, in order to provide an upper bound to \eqref{e:hratio}, we need to bound $\frac{C_d+2}{2}\sum_i\left((m_i+k_i)(m_i+k_i-1)-m_i(m_i-1)\right)$ from below and $\sum_{i=1}^{2^d}\Big(\sum_{r=1}^{m_i+k_i-1}\gamma(r)-\sum_{r=1}^{m_i-1}\gamma(r)\Big)$ from above.

 To this end, let $\varepsilon_i:=m_i-\lfloor 2^{-d}n\rfloor$. Since $\sum_i k_i=0$ and $\varepsilon_i\in \{0, 1\}$ for each $i$, 
\[|\sum_i m_ik_i|= |\sum_i \varepsilon_ik_i|\leq \sum_i|k_i|\,.\]
Thus,
\begin{eqnarray*}
&&\frac{C_d+2}{2}\sum_i\left((m_i+k_i)(m_i+k_i-1)-m_i(m_i-1)\right)\\
&=&\frac{C_d+2}{2}\left(\sum_i k_i^2+2\sum_i m_ik_i\right)=\frac{C_d+2}{2}(\sum_i k_{i}^2+R)\,,
\end{eqnarray*}
where $|R|\leq 2\sum_i |k_i|$.

Next, by Lemma \ref{l:d-2byd}, for all $i\in \{1,2,\ldots,2^d\}$ such that $k_i\geq 0$ and all $z\in \{0,1,\ldots,k_i\}$
\begin{equation}\label{e:thor1}
\gamma(\lfloor 2^{-d}n\rfloor+z)\leq \gamma(\lfloor 2^{-d}n\rfloor)+4k_i^{\frac{d-2}{d}}\,.
\end{equation}
And for all $i$ such that $k_i<0$,  and all $z\in \{0,1,\ldots,|k_i|\}$
\begin{equation}\label{e:thor2}
\gamma(\lfloor 2^{-d}n\rfloor)\leq \gamma(\lfloor 2^{-d}n\rfloor-z)+4|k_i|^{\frac{d-2}{d}}\,.
\end{equation}
Now, without loss of generality, we assume that the first $s$ many of the $k_{i}$'s are non negative, and the rest are negative. Then,
 \begin{align*}
&\sum_{i=1}^{2^d}\Big(\sum_{r=1}^{m_i+k_i-1}\gamma(r)-\sum_{r=1}^{m_i-1}\gamma(r)\Big)
= \sum_{i=1}^{s}\Big(\sum_{r=m_i}^{m_i+k_i-1}\gamma(r)\Big)-\sum_{i=s+1}^{2^d}\Big(\sum_{r=m_i+k_i}^{m_i-1}\gamma(r)\Big)\\
&=\sum_{i=1}^{s}\Big(\sum_{r=\lfloor 2^{-d}n\rfloor+\varepsilon_i}^{\lfloor 2^{-d}n\rfloor+\varepsilon_i+k_i-1}\gamma(r)\Big)-\sum_{i=s+1}^{2^d}\Big(\sum_{r=\lfloor 2^{-d}n\rfloor+\varepsilon_i+k_i}^{\lfloor 2^{-d}n\rfloor+\varepsilon_i-1}\gamma(r)\Big)\\
&\leq \sum_{i=1}^s k_i\Big(\gamma(\lfloor 2^{-d}n\rfloor)+4k_i^{\frac{d-2}{d}}\Big)+\sum_{i=s+1}^{2^d}|k_i|\Big(-\gamma(\lfloor 2^{-d}n\rfloor)+4|k_i|^{\frac{d-2}{d}}\Big)\\
&= \gamma(\lfloor 2^{-d}n\rfloor)\Big(\sum_{i=1}^{2^d}k_i\Big)+ \sum_i4|k_{i}|^{1+\frac{d-2}{d}}=4\sum_i|k_{i}|^{1+\frac{d-2}{d}}\,.
 \end{align*}
 Here, the inequality in the third line follows from \eqref{e:thor1} and \eqref{e:thor2} and the last equality follows since $\sum k_i=0$.
Therefore
\begin{align}\label{e:diffL}
\sum_{i=1}^{2^d}(L_{m_i+k_i}-L_{m_i})&\geq \frac{C_d+2}{2}\sum k_i^2-C_0\sum_i|k_{i}|^{1+\frac{d-2}{d}}-C_0'\sum_i|k_i|\nonumber\\
&\geq \frac{C_d+2}{2}\sum k_i^2-C_1\sum_i|k_{i}|^{1+\frac{d-2}{d}}\,,
\end{align}
where $C_0,C_0',C_1$ are universal constants.
Plugging this in \eqref{e:hratio}, we get
\begin{eqnarray*}
&&\frac{h(m_1+k_1,\ldots,m_{2^d}+k_{2^d})}{h(m_1,\ldots,m_{2^d})}\\
&\leq& \exp\left(-2^{d-3}(C_d+2)\beta\sum_{i=1}^{2^d}k_i^2+ 2^{d-2}C_1\beta\sum_{i=1}^{2^d}|k_i|^{1+\frac{d-2}{d}}+C(\beta)\beta \log^6 n\sum_{i=1}^{2^d}|k_i|\right)\,.
\end{eqnarray*}

The remainder of the proof is quite similar to the proof of \cite[Lemma $2.5$]{cha}. We include the details for completeness.
By definition it follows that 
\[\frac{{g}(m_1+k_1,\ldots,m_{2^d}+k_{2^d})}{{g}(m_1,\ldots,m_{2^d})}\leq \exp\left(\beta\sum_{i=1}^{2^d}(2|k_i|+k_i^2)\right).\]
Therefore,
\begin{align*}
&\frac{\P(N_1=m_1+k_1,\ldots,N_{2^d}=m_{2^d}+k_{2^d})}{\P(N_1=m_1,\ldots,N_{2^d}=m_{2^d})}\\
&= \frac{f(m_1+k_1,\ldots,m_{2^{d}}+k_{2^d})}{f(m_1,\ldots,m_{2^d})}\cdot\frac{g(m_1+k_1,\ldots,m_{2^{d}}+k_{2^d})}{g(m_1,\ldots,m_{2^d})}\cdot\frac{h(m_1+k_1,\ldots,m_{2^{d}}+k_{2^d})}{h(m_1,\ldots,m_{2^d})}\\
&\leq \frac{f(m_1+k_1,\ldots,m_{2^{d}}+k_{2^d})}{f(m_1,\ldots,m_{2^d})}\exp\left(-C_2\beta\sum_{i=1}^{2^d}k_i^2+C_1(\beta)\beta\log^6 n\sum_{i=1}^{2^d}|k_i|+C\beta\sum_{i=1}^{2^d}|k_i|^{1+\frac{d-2}{d}}\right)\,,
\end{align*}
where $C_2=2^{d-3}(C_d+2)-1>1$ for $d\geq 3$, $C_1(\beta)=C(\beta)+2$ is a decreasing function of $\beta$ and $C=2^{d-2}C_1$.

Thus, for $C_2(\beta):=\max\left\{\frac{2^{d+1}C_1(\beta)}{C_2-1},(2^dC)^{\frac{d}{2}},1\right\}$ and $C_2':=\frac{C_2-1}{2}$, if
$\max_{1\leq i\leq 2^d}|k_i|\geq C_2(\beta)(\log^6 n\vee1)$, (where $(a\vee b)$ denotes the maximum of $a$ and $b$),
then
\begin{align*}
&-C_2\beta\sum_{i=1}^{2^d}k_i^2+C_1(\beta)\beta\log^6 n\sum_{i=1}^{2^d}|k_i|+C\beta\sum_{i=1}^{2^d}|k_i|^{1+\frac{d-2}{d}}\\
&\leq -C_2\beta\max|k_i|^2+2^dC_1(\beta)\beta\log^6n\max|k_i|+2^dC\beta\max|k_i|^{1+\frac{d-2}{d}}\\
&\leq -(C_2-1)\beta\max|k_i|^2+2^dC_1(\beta)\beta\log^6n\max|k_i|\\
&\leq -\frac{(C_2-1)\beta}{2}\max|k_i|^2
\leq -C_2'\beta\log^{12}n\,.
\end{align*}
Here, the second inequality follows because $2^dC\leq C_2(\beta)^{\frac{2}{d}}\leq \max|k_i|^{\frac{2}{d}}$, the third inequality follows due to
$$2^dC_1(\beta)\log^6n\leq \frac{C_2-1}{2}C_2\log^6n\leq \frac{C_2-1}{2}\max|k_i|\,,$$
and the last inequality follows because $\max|k_i|^2\geq \log^{12}n$, since $C_2(\beta)\geq 1$. Therefore, in this case,
\[\frac{\P(N_1=m_1+k_1,\ldots,N_{2^d}=m_{2^d}+k_{2^d})}{\P(N_1=m_1,\ldots,N_{2^d}=m_{2^d})}\leq \frac{f(m_1+k_1,\ldots,m_{2^{d}}+k_{2^d})}{f(m_1,\ldots,m_{2^d})}e^{-C_2'\beta\log^{12}n}\,.\]
If $A$ denotes the set of all $(n_1,\ldots,n_{2^d})$ such that each $n_i$ is a non negative integer, $n_1+\ldots+n_{2^d}=n$ and
\[\max_{1\leq i\leq 2^d}|n_i-m_i|\geq C_2(\beta)(\log^6n\vee 1)\,, \]
then, using the multinomial formula and Stirling's approximation, which yields $$\sum_{k_1,\ldots, k_{2^d}} f(m_1+k_1,m_2+k_2,\ldots,m_{2^{d}}+k_{2^{d}})\le C_2''n^{2^{d-1}}f(m_1,m_2,\ldots,m_{2^{d}})$$
for some positive constant $C_2''$, we get,
\[\P((N_1,\ldots,N_{2^d})\in A)\leq \sum_{(n_1,\ldots,n_{2^d})\in A}\frac{\P(N_1=m_1+k_1,\ldots,N_{2^d}=m_{2^d}+k_{2^d})}{\P(N_1=m_1,\ldots,N_{2^d}=m_{2^d})}\leq C_2''n^{2^{d-1}}e^{-C_2'\beta\log^{12}n}\,.\]
Therefore for each $i$,
\begin{align*}
\mbox{Var}(N_i)&\leq \E(N_i-m_i)^2
\leq C_2(\beta)^2(\log^{12}n\vee 1)+n^2\P((N_1,\ldots,N_{2^d})\in A)\\
&\leq C_2(\beta)^2(\log^{12}n\vee 1)+C_2''n^{2^{d-1}+2}e^{-C_2'\beta\log^{12}n}\,.
\end{align*}
Thus, since $C_2(\beta)$ is a decreasing function of $\beta$,
we get $\mbox{Var}(N_i)\leq K(\beta)\log^{12}n\,,$
where $K(\beta)$ is a deceasing function of $\beta$. The proof for any dyadic box is similar with some new ingredients. This is presented in  Lemma \ref{uncondvar} later.
\qed

\vspace{.2in}

We now have all the ingredients to finish the proof of Theorem \ref{marigid}.
\section{Sub-extensive fluctuations for sets with smooth boundaries}\label{sketch}
As indicated before, given the above inputs, the proof of Theorem \ref{marigid} follows using Chatterjee's arguments closely. However for completeness we sketch the main steps referring the interested reader to \cite{cha} for more details.  The key object to analyze is a Doob martingale adapted to the filtration given by the various levels of the $2^{d}-$ary tree.  
Recall from \eqref{tree}
that $\cD_k$ denotes the set of all dyadic sub-cubes at level $k$. For any nonempty open set $U\subseteq [0,1)^d$ with a regular boundary, 
we will now need to define a class of sets:
$$\cU:=\{D\in \cD:\, D\subseteq U,\,\,  D'\not\subset U\}\,\, \text{and }\cU_j:=\cU\cap \cD_j$$ where $D'$ is the parent of $D.$ 
Also $$\cV_j:=\{D\in \cD_j:\, D\not\subset U,\,\,  D\not\subset U^c\},$$
i.e., that intersect both $U$ and $U^c$. For any dyadic cube $D$, let $p(D)=\frac{\mathrm{Leb}(D\cap U)}{\mathrm{Leb}(D)}$. We now explicitly describe the Doob Martingale $M_j=\E(N(U)\mid \cF_j)$ where $\cF_j$ is the $\sigma$-algebra generated by the random variables $\{N(D):D\in \cD_j\}$, by defining  $M_0=\mathrm{Leb}(U)n$ and for every $j\geq 1$, setting
\begin{equation}\label{doob1}
M_j:=\sum_{i=0}^j\sum_{D\in \cU_i}N(D)+\sum_{D\in \sV_j}p(D)N(D)\,.
\end{equation} (see \cite[Lemma 2.11 ]{cha} for a formal proof that this is indeed a martingale).
We now proceed to proving the upper bound in Theorem \ref{marigid}. The proof is slightly simpler than the one appearing in \cite{cha} since the strong hyperuniformity bound established in Theorem \ref{keyrigid} allows us to be more crude in our arguments at the cost of paying some logarithmic factors in the statement of the upper bound in Theorem \ref{marigid}.
We start with a series of simple but useful lemmas. 
\begin{lem}\cite[Lemma 2.7]{cha}\label{condvar} For any $D\in \cD_j,$ and $D'$ a child of $D$ we have 
\begin{align*}
\E(D'\mid \cF_{j})&=\frac{N(D)}{2^{d}},\\
\Var(N(D')\mid \cF_{j})&\le K(\beta) \max(\log^{12}(N(D)),1),
\end{align*}
where $K(\beta)$ is a decreasing function of $\beta.$
\end{lem}
\begin{proof} The proof of expectation is immediate by exchangeability of particles. The variance bound follows by using the hierarchical structure of the model and the already proven part of Theorem \ref{keyrigid}.
\end{proof}

\begin{lem}\cite[Lemma 2.8]{cha}(Unconditional variance)\label{uncondvar} For any $D\in \cD,$ 
$\E(N(D))={\rm Leb}(D)n$ and 
$$\Var(N(D))\le K(\beta)[\log^{12}(\E(N(D))+1],$$
where $K(\beta)$ is a decreasing function of $\beta.$
In particular for any $D$ using $\E(N(D))\le n$ we get $$\Var(N(D))\le O(K(\beta)\log^{12}(n)),$$ and for any $D$ such that $\E(N(D))=O(1)$ we get $\Var(N(D))=O(1).$ 
\end{lem}
The first half of the result completes the proof of the final statement  of Theorem \ref{keyrigid}.

\begin{proof} The statement about expectation is just a consequence of the fact that each point is marginally uniformly distributed on $[0,1)^d$, so we focus on the variance. Let $D\in \cD_j$ and $D' \in \cD_{j-1}$ be its parent. Then using the form of $\E(N(D)\mid \cF_{j-1})$ and the previous lemma we get,
\begin{align*}
\E(N^2(D))&=\E(\Var(N(D)\mid \cF_{j-1}))+2^{-2d}\E(N^2(D'))\\
&\le K(\beta)\E\log^{12}(N(D')+c)+2^{-2d}\E(N^2(D'))\\
&\le K(\beta)\log^{12}(\E(N(D'))+c)+2^{-2d}\E(N^2(D'))
\end{align*}
where $c>0$ is chosen such that $\log^{12}(x+c)$ is concave for $x>0.$  
Iterating this we obtain
\begin{align*}
\E(N^2(D))\le K(\beta)\log^{12}(\E(N(D)+c))(1+\frac{\log^{12}(2^{d})}{2^{2d}}+\frac{\log^{12}(2^{2d})}{2^{4d}}+\ldots)+2^{-2dj}n^2
\end{align*}
Since $\E(N(D))=2^{-dj}n$ we obtain $\Var(N(D))=K(\beta)\log^{12}(\E(N(D))+c).$
\end{proof}
The next lemma bounds the variance of the martingale $M_j$ defined in \eqref{doob1}.
\begin{lem}\cite[Lemma 2.12]{cha}\label{martvar} Let $k$ be the smallest integer such that $2^{dk}\ge n.$ For all $j\le k$
\begin{align*}
\Var(M_j)\le O(\log^{12}(n)n^{\frac{d-1}{d}})+\Var (M_{j-1}).
\end{align*}
In particular this implies $\Var(M_k)=O(\log^{13}(n) n^{\frac{d-1}{d}}).$
\end{lem}
\begin{proof}It is easy to check (decomposition of variance property of martingales) that 
$$\Var(M_j)=\E(\Var(M_j\mid \cF_{j-1}))+\Var(M_{j-1}).$$
Now by \eqref{doob1},
\begin{align*}
\Var(M_j\mid \cF_{j-1})&=\Var\left(\sum_{D\in \cU_j\cup \cV_j}p(D)N(D)\right)
=\sum_{D,D' \in \cU_j\cup \cV_j}p(D)p(D') \Cov(N(D),N(D')\mid \cF_{j-1}).
\end{align*}
Now conditional on $\cF_{j-1},$ $N(D)$ and $N(D')$ are independent, if $D$ and $D'$ are not siblings.  Otherwise Cauchy-Schwarz inequality and Lemma \ref{condvar} yields that 
$\Cov(N(D),N(D')\mid \cF_{j-1})\le K(\beta)\log^{12}(n).$
Thus 
\begin{align*}
\Var(M_j\mid \cF_{j-1})&\le \sum_{\begin{subarray}{c}D,D' \in \cU_j\cup \cV_j\\
D,D' \text{ are siblings}
\end{subarray}}p(D)p(D') \Cov(N(D),N(D')\mid \cF_{j-1})\\
& \le O(\log^{12}(n))\Big|\cU_j\cup \cV_j\Big|=O(\log^{12}(n))n^{\frac{d-1}{d}}).
\end{align*}
where the last inequality uses that the number of siblings of any $D\in \cD$ is bounded by $2^d$ and the observation that $|\cU_j\cup \cV_j|\le |\cU_k\cup \cV_k|=O(n^{\frac{d-1}{d}}),$ for all $j\le k,$ by our choice of $k,$ and the hypothesis of regularity on $U,$ (for the precise geometric details see \cite[Pg 21]{cha}).
\end{proof}

Using the above preparation we can now finish the proof of the upper bound in Theorem \ref{marigid}.
\begin{proof}[Proof of Upper bound] Let $k$ be as in the previous lemma. 
Using the decomposition $$U=\left(\bigcup_{j=0}^k\cU_j\right)\cup \left(\bigcup_{D\in \cV_k}(D\cap U)\right),$$ we get $$N(U)=\sum_{j=0}^k\sum_{D\in \cU_j}N(D)+\sum_{D\in \cV_k}N(D\cap U).$$
Since $\E(N(U)\mid \cF_k)=M_k$ we get 
\begin{align}\label{decom98}
\Var(N(U))=\E(\Var(N(U)\mid \cF_k))+\Var(M_k).
\end{align}
Since $\sum_{j=0}^k\sum_{D\in \cU_j}N(D)$ is measurable with respect to $\cF_k$ we have,
\begin{align}\label{decom12}
\Var(N(U)\mid \cF_k)&=\Var\left(\sum_{D\in \cV_{k}}N(D\cap U)\Bigg\vert   \cF_k\right)
=\sum_{D\in \cV_k}\Var(N(D\cap U)\mid \cF_k)\\
\nonumber
&\le \sum_{D\in \cV_k}\E(N(D\cap U)^2\mid \cF_k)\le \sum_{D\in \cV_k}N(D)\E(N(D\cap U)\mid \cF_k)
\le \sum_{D\in \cV_k}p(D)N(D)^2.
\end{align}
Now for any $D$ in the sum above, since $\E(N(D))=O(1)$ by our choice of $k,$ using $\E(N(D)^2)=O(1)$ (Lemma \ref{uncondvar}) and $p(D)\le 1$ for all $D\in \cD,$ we get $$\E(\Var(N(U)\mid \cF_k))\le O(|\cV_k|)=O(n^{\frac{d-1}{d}}).$$
Thus by \eqref{decom98} and Lemma \ref{martvar} it follows that $\Var(N(U))=O(\log^{13}(n)n^{\frac{d-1}{d}}).$

\end{proof}

We now provide the key steps of the proof of the lower bound in Theorem \ref{marigid}. 

\begin{proof}[Proof of lower bound] Recall the sketch from Figure \ref{fig5}. First, one can observe that the geometric lemmas \cite[Lemmas 2.14, 2.15, 2.16]{cha}  with trivial modifications work in dimension $d$ to imply the following statement analogous to \cite[Lemmas 2.17, 2.18]{cha}: For any set $U$ satisfying the conditions in Theorem \ref{marigid}, there exists $c,K_1>0$ and some $j_1\geq 1$ depending only on $U$, such that for any $j\geq j_1$, there exists a set of at least $K_12^{(d-1)j}$ cubes $D\in \cD_j$ satisfying
\begin{equation}\label{e:intersect}
c\leq \frac{\mathrm{Leb}(D\cap U)}{\mathrm{Leb}(D)}\leq 1-c\,,
\end{equation}
and the diameter of the union of all these cubes is at most $\diam(U)/3$.
Also, the bound in \cite[Lemma 2.19]{cha} on the probability of a dyadic cube having multiple points,  in dimension $d$ becomes the following. For any $n\geq 1,\beta>0,j\geq 0$ and $D\in \cD_j$,
\begin{equation}\label{e:nd2}
\P(N(D)\geq 2)\leq \exp\left(-2^{(d-2)(j+1)}\beta+(C_d+2){n\choose 2}\right)\,.
\end{equation}
Now, we choose $k$ such that $2^{dk}=O(n)$. Then, using Lemma \ref{uncondvar}, one has $\E(N(D)^2)=O(1)$ for any $D\in \cD_k$. 
Again, for any $D\in\cD_k$, and $j>k,$ let $\cD_j(D)$ be the set of elements of $\cD_j$ which are descendants of $D$. Now, for  $m$ a large enough fixed constant,  using \eqref{e:nd2}, we get a $j>k$ with $j-k=O(1)$ such that the event $E$ defined as
\[E=\{N(D)\leq m, N(D')\geq 2, \mbox{ for some } D\in \cD_k, D'\in \cD_j(D)\}\,,\]
has a probability that decays with $n$. 
Given such a $j$, we get a set $\sD'\subseteq \cD_j$ such that $|\sD'|=\Theta(n^{\frac{d-1}{d}})$ and the elements satisfy \eqref{e:intersect}. Let $\sD$ denote the set of ancestors of the elements of $\sD'$in $\cD_k$. Now define
\[q:=\frac{|\{D\in\sD:0<N(D)\leq m\}|}{|\sD|}\,.\]
To lower bound $q$,  one observes that a lower bound on 
$\frac{|\{D\in\sD:0<N(D)\}|}{|\sD|}$ and an upper bound on $\frac{|\{D\in\sD:N(D)\leq m\}|}{|\sD|}$ follows by the 
second moment method or Paley-Zygmund inequality (to make the second moment method go through, one uses $\E(N(D)^2)=O(1)$ for $D\in \cD_k$ which was proved in Lemma \ref{uncondvar}) and Markov's inequality respectively. Combining we obtain, 
\[\P\left(q\geq \frac{C}{m}\right)\geq \frac{1}{C}\,,\]
for some $C>1$. 
Now, if $\sD_0$ is the set of all elements of $\sD'$ that are contained in some $D\in \sD$ such that $0<N(D)\leq m$, then  the fact that $q\geq \Theta(1)$ with a probability bounded away from zero, implies that 
\[\P(|\sD_0|\geq Kn^{\frac{d-1}{d}})\geq c>0\,,\]
for some constant $K>0$.
Next define 
\[\sC:=\{D\in \sD_0: N(D)=1\}\,.\]
Because of the bound on $|\sD_0|$ and the fact that $E^c$ is a high probability event, we get
\begin{equation}\label{e:lwb}
\P(|\sC|\geq Kn^{\frac{d-1}{d}})\geq c>0\,.
\end{equation}
Finally, define
$$M:=\sum_{D\in \sC} N(D\cap U)\,.$$
Since the random variables $\{N(D\cap U):D\in \cD_j\}$ are independent given $\cF_j$, the conditional distributions of $\{N(D\cap U):D\in \sC\}$ given $\cF_j$ are independent Bernoulli random variables with probabilities $p(D)$ satisfying $p(D)\in [c,1-c]$. Thus, for any interval $I$,
\begin{equation}\label{anticonc}
\P(M\in I|\cF_j)\leq O\left(\frac{|I|}{\sqrt{|\cC^*|}}\right)\,.
\end{equation}
Since
\[N(U)=\sum_{D\in \cD_j}N(D\cap U)=\sum_{D\in \cD_j\setminus \sC}N(D\cap U)+M\,,\]
and conditional on $\cF_j$, the two terms on the RHS are independent, the
anti-concentration statement in Theorem \ref{marigid} now follows easily from here using \eqref{e:lwb} and \eqref{anticonc}.
\end{proof}

\section{Fluctuations of smooth linear statistics}\label{posmooth} In this section we prove Theorem \ref{smoothsharp} by a similar argument as in the proof of Theorem \ref{marigid} with the only difference being, we consider a slightly different  martingale sequence $\{W_{j}\}_{j\ge 0}$ given by $W_j=\E(X(f)\mid \cF_{j}).$ Note the following alternate description: For any $D\in \cD,$ let $f(D)$ be the average of $f$ over $D$ and let $f_j$ be the function that is equal to $f(D)$ for all $D\in \cD_j.$ Then $W_{j}=X(f_j).$ Then as in the proof of Theorem \ref{marigid}, it follows that, for any $k,$
\begin{equation}\label{decomsmooth}
\Var(X(f))=\E(\Var(X(f))\mid \cF_k)+\sum_{j=1}^{k}\E(\Var(X(f_j))\mid \cF_{j-1}).
\end{equation}
We will take $k$ to be the minimum value such that $2^{dk}\ge n.$ Also for any $D\in \cD$, let $c(D)$ denote the set of children of $D.$ Now for any $j$ observe that 
\begin{align*}
&\Var(X(f_j)\mid \cF_{j-1})=\Var\left(\sum_{D\in\cD_j}f(D)N(D)\bigg|\cF_{j-1}\right)
=\sum_{D\in \cD_{j-1}}\Var\left(\sum_{D'\in c(D)}f(D')N(D')\bigg| \cF_{j-1}\right)\\
&=\sum_{D\in \cD_{j-1}}\E\left(\left(\sum_{D'\in c(D)}f(D')N(D')-f(D)N(D)\right)^2\bigg|\cF_{j-1}\right), \text{ since $\sum_{D'\in c(D)}f(D')=2^{d}f(D).$}
\end{align*}
Now we use the observation that for any $D\in\cD_{j-1},$
$$\sum_{D'\in c(D)}f(D')N(D')-f(D)N(D)=\sum_{D'\in c(D)}(f(D')-f(D))\left(N(D')-\frac{N(D)}{2^{d}}\right)$$ and that $f$ has Lipschitz constant $L$ and ${\rm{diam}}(D)=O(\frac{1}{2^{j}}),$ to conclude that $|f(D)-f(D')|\le \frac{O(L)}{2^{j}}$ and hence
\begin{align*}
&\sum_{D\in \cD_{j-1}}\E\left(\left(\sum_{D'\in c(D)}(f(D')-f(D))\left(N(D')-\frac{N(D)}{2^{d}}\right)\right)^2\bigg|\cF_{j-1}\right)\\
&\le \frac{O(L^{2})}{4^{j}} \sum_{D\in \cD_{j-1}}\E\left(\left(\sum_{D'\in c(D)}\left|N(D')-\frac{N(D)}{2^{d}}\right|\right)^2\bigg|\cF_{j-1}\right)\\
&\le\frac{O(L^{2})}{4^{j}} \sum_{D\in \cD_{j-1}}\sum_{D'\in c(D)}\Var(N(D')\mid \cF_{j-1})\\
&\le\frac{O(L^{2})}{4^{j}}\sum_{D\in \cD_{j-1}}K(\beta)\max(\log^{12}(N(D)),1),\,\text{ by Lemma \ref{condvar}.}
\end{align*}
Above the constants in the $O(\cdot)$ notation are universal but change from line to line. 
We now prove a similar bound on $\Var(X(f)\mid \cF_k).$ Let for any $D\in \cD_{k},$ $s(D)=\sum_{X_j\in D}f(X_j).$ Thus $X(f)=\sum_{D\in \cD_k}s(D).$
Using similar arguments as above we obtain, \begin{align*}
\Var(X(f)\mid \cF_{k})&=\sum_{D\in \cD_k} \Var(s(D)\mid \cF_{k})= \sum_{D\in \cD_k}\E\left((s(D)-f(D)N(D))^2\bigg| \cF_k\right)\\
& \le \sum_{D\in \cD_k}\frac{O(L^2)}{4^k}\E(N(D)^2 | \cF_k),\end{align*}
where the last inequality uses the Lipschitz nature of $f.$ 
Thus putting everything together using \eqref{decomsmooth} and Lemma \ref{uncondvar}, our choice of $k,$ and that $|\cD_j|=2^{dj},$ we get that 
\begin{align*}
\Var(X(f))&=\E(\Var(X(f)\mid \cF_k))+\sum_{j=1}^{k}\E(\Var(X(f_j)\mid \cF_{j-1}))\\
&= O(K(\beta)L^2)\sum_{j=1}^k{\log^{12}\left(\frac{n}{2^{dj}}\right)}2^{(d-2)j}
= O(K(\beta)L^2)\sum_{j=1}^k{(k-j)^{12}}2^{(d-2)j}\\
&=O(K(\beta)L^2 2^{(d-2)k})=O(K(\beta)L^2n^{\frac{d-2}{d}}).
\end{align*}
Hence the proof is complete. \qed

 \bibliography{Hierarchical-coulomb}

\begin{thebibliography}{10}

\bibitem{agz10}
Greg~W. Anderson, Alice Guionnet, and Ofer Zeitouni.
\newblock {\em An introduction to random matrices}, volume 118 of {\em
  Cambridge Studies in Advanced Mathematics}.
\newblock Cambridge University Press, Cambridge, 2010.

\bibitem{bardenethardy16}
Rémi Bardenet and Adrien Hardy.
\newblock Monte carlo with determinantal point processes.
\newblock 2016.
\newblock {\it arXiv preprint arXiv:1605.00361}.

\bibitem{bbny16}
Roland Bauerschmidt, Paul Bourgade, Miika Nikula, and Horng-Tzer Yau.
\newblock The two-dimensional coulomb plasma: quasi-free approximation and
  central limit theorem, 2016.
\newblock {\it arXiv preprint arXiv:1609.08582}.

\bibitem{bbny15}
Roland Bauerschmidt, Paul Bourgade, Miika Nikula, and Horng-Tzer Yau.
\newblock Local density for two-dimensional one-component plasma.
\newblock {\em Communications in Mathematical Physics}, 356(1):189--230, Nov
  2017.

\bibitem{BeltranHardy}
Carlos Beltr\'{a}n and Adrien Hardy.
\newblock Energy of the {C}oulomb {G}as on the {S}phere at {L}ow {T}emperature.
\newblock {\em Arch. Ration. Mech. Anal.}, 231(3):2007--2017, 2019.

\bibitem{bz98}
G\'{e}rard Ben~Arous and Ofer Zeitouni.
\newblock Large deviations from the circular law.
\newblock {\em ESAIM Probab. Statist.}, 2:123--134, 1998.

\bibitem{benfattoetal86}
G.~Benfatto, G.~Gallavotti, and F.~Nicol\`o.
\newblock The dipole phase in the two-dimensional hierarchical {C}oulomb gas:
  analyticity and correlations decay.
\newblock {\em Comm. Math. Phys.}, 106(2):277--288, 1986.

\bibitem{benfattorenn92}
G.~Benfatto and J.~Renn.
\newblock Nontrivial fixed points and screening in the hierarchical
  two-dimensional {C}oulomb gas.
\newblock {\em J. Statist. Phys.}, 67(5-6):957--980, 1992.

\bibitem{borodinsinclair09}
A.~Borodin and C.~D. Sinclair.
\newblock The {G}inibre ensemble of real random matrices and its scaling
  limits.
\newblock {\em Comm. Math. Phys.}, 291(1):177--224, 2009.

\bibitem{bey12}
Paul Bourgade, L\'{a}szl\'{o} Erd\H{o}s, and Horng-Tzer Yau.
\newblock Bulk universality of general {$\beta$}-ensembles with non-convex
  potential.
\newblock {\em J. Math. Phys.}, 53(9):095221, 19, 2012.

\bibitem{bey14}
Paul Bourgade, L\'{a}szl\'{o} Erd\H{o}s, and Horng-Tzer Yau.
\newblock Universality of general {$\beta$}-ensembles.
\newblock {\em Duke Math. J.}, 163(6):1127--1190, 2014.

\bibitem{bey14b}
Paul Bourgade, L\'{a}szl\'{o} Erd\"{o}s, and Horng-Tzer Yau.
\newblock Edge universality of beta ensembles.
\newblock {\em Comm. Math. Phys.}, 332(1):261--353, 2014.

\bibitem{byy14a}
Paul Bourgade, Horng-Tzer Yau, and Jun Yin.
\newblock Local circular law for random matrices.
\newblock {\em Probab. Theory Related Fields}, 159(3-4):545--595, 2014.

\bibitem{byy14b}
Paul Bourgade, Horng-Tzer Yau, and Jun Yin.
\newblock The local circular law {II}: the edge case.
\newblock {\em Probab. Theory Related Fields}, 159(3-4):619--660, 2014.

\bibitem{BPT}
D.~P. Bourne, M.~A. Peletier, and F.~Theil.
\newblock Optimality of the triangular lattice for a particle system with
  {W}asserstein interaction.
\newblock {\em Comm. Math. Phys.}, 329(1):117--140, 2014.

\bibitem{chafaietal14}
Djalil Chafa\"{i}, Nathael Gozlan, and Pierre-Andr\'{e} Zitt.
\newblock First-order global asymptotics for confined particles with singular
  pair repulsion.
\newblock {\em Ann. Appl. Probab.}, 24(6):2371--2413, 2014.

\bibitem{chafai16}
Djalil Chafa\"{i}, Adrien Hardy, and Myl\`ene Ma\"{i}da.
\newblock Concentration for {C}oulomb gases and {C}oulomb transport
  inequalities.
\newblock {\em J. Funct. Anal.}, 275(6):1447--1483, 2018.

\bibitem{cha}
Sourav Chatterjee.
\newblock Rigidity of the three-dimensional hierarchical coulomb gas, 2017.
\newblock {\it arXiv preprint arXiv:1708.01965}.

\bibitem{costinlebowitz95}
Ovidiu Costin and Joel~L. Lebowitz.
\newblock Gaussian fluctuation in random matrices.
\newblock {\em Phys. Rev. Lett.}, 75:69--72, Jul 1995.

\bibitem{diaconisevans01}
Persi Diaconis and Steven~N. Evans.
\newblock Linear functionals of eigenvalues of random matrices.
\newblock {\em Trans. Amer. Math. Soc.}, 353(7):2615--2633, 2001.

\bibitem{dimock90}
J~Dimock.
\newblock The kosterlitz-thouless phase in a hierarchical model.
\newblock {\em Journal of Physics A: Mathematical and General},
  23(7):1207--1215, apr 1990.

\bibitem{dyson53}
Freeman~J. Dyson.
\newblock The dynamics of a disordered linear chain.
\newblock {\em Phys. Rev.}, 92:1331--1338, Dec 1953.

\bibitem{dyson69}
Freeman~J. Dyson.
\newblock Existence of a phase-transition in a one-dimensional {I}sing
  ferromagnet.
\newblock {\em Comm. Math. Phys.}, 12(2):91--107, 1969.

\bibitem{erdhos2012}
L{\'a}szl{\'o} Erd{\H{o}}s, Horng-Tzer Yau, and Jun Yin.
\newblock Rigidity of eigenvalues of generalized wigner matrices.
\newblock {\em Advances in Mathematics}, 229(3):1435--1515, 2012.

\bibitem{forrester10}
Peter~J. Forrester.
\newblock {\em Random matrices, log-gases and the Calogero-Sutherland model},
  volume Volume 1 of {\em MSJ Memoirs}, pages 97--181.
\newblock The Mathematical Society of Japan, Tokyo, Japan, 1998.

\bibitem{gl17}
Subhro Ghosh and Joel Lebowitz.
\newblock Number rigidity in superhomogeneous random point fields.
\newblock {\em J. Stat. Phys.}, 166(3-4):1016--1027, 2017.

\bibitem{ghosh15}
Subhroshekhar Ghosh.
\newblock Determinantal processes and completeness of random exponentials: the
  critical case.
\newblock {\em Probab. Theory Related Fields}, 163(3-4):643--665, 2015.

\bibitem{ghosh16}
Subhroshekhar Ghosh.
\newblock Palm measures and rigidity phenomena in point processes.
\newblock {\em Electron. Commun. Probab.}, 21:Paper No. 85, 14, 2016.

\bibitem{gl17b}
Subhroshekhar Ghosh and Joel~L. Lebowitz.
\newblock Fluctuations, large deviations and rigidity in hyperuniform systems:
  a brief survey.
\newblock {\em Indian J. Pure Appl. Math.}, 48(4):609--631, 2017.

\bibitem{ghoshperes17}
Subhroshekhar Ghosh and Yuval Peres.
\newblock Rigidity and tolerance in point processes: {G}aussian zeros and
  {G}inibre eigenvalues.
\newblock {\em Duke Math. J.}, 166(10):1789--1858, 2017.

\bibitem{gz16}
Subhroshekhar Ghosh and Ofer Zeitouni.
\newblock Large deviations for zeros of random polynomials with i.i.d.
  exponential coefficients.
\newblock {\em Int. Math. Res. Not. IMRN}, (5):1308--1347, 2016.

\bibitem{guidimarchetti01}
Leonardo~F. Guidi and Domingos H.~U. Marchetti.
\newblock Renormalization group flow of the two-dimensional hierarchical
  {C}oulomb gas.
\newblock {\em Comm. Math. Phys.}, 219(3):671--702, 2001.

\bibitem{hardy12}
Adrien Hardy.
\newblock A note on large deviations for 2{D} {C}oulomb gas with weakly
  confining potential.
\newblock {\em Electron. Commun. Probab.}, 17:no. 19, 12, 2012.

\bibitem{HR}
Raymond~C. Heitmann and Charles Radin.
\newblock The ground state for sticky disks.
\newblock {\em J. Statist. Phys.}, 22(3):281--287, 1980.

\bibitem{holroydsoo13}
Alexander~E. Holroyd and Terry Soo.
\newblock Insertion and deletion tolerance of point processes.
\newblock {\em Electron. J. Probab.}, 18:no. 74, 24, 2013.

\bibitem{houghetal09}
J.~Ben Hough, Manjunath Krishnapur, Yuval Peres, and B\'{a}lint Vir\'{a}g.
\newblock {\em Zeros of {G}aussian analytic functions and determinantal point
  processes}, volume~51 of {\em University Lecture Series}.
\newblock American Mathematical Society, Providence, RI, 2009.

\bibitem{jancovicietal93}
B.~Jancovici, J.~L. Lebowitz, and G.~Manificat.
\newblock Large charge fluctuations in classical {C}oulomb systems.
\newblock {\em J. Statist. Phys.}, 72(3-4):773--787, 1993.

\bibitem{joh19}
Kurt Johansson et~al.
\newblock On fluctuations of eigenvalues of random hermitian matrices.
\newblock {\em Duke mathematical journal}, 91(1):151--204, 1998.

\bibitem{joh20}
Kurt Johansson, Gaultier Lambert, et~al.
\newblock Gaussian and non-gaussian fluctuations for mesoscopic linear
  statistics in determinantal processes.
\newblock {\em The Annals of Probability}, 46(3):1201--1278, 2018.

\bibitem{kappeleretal91}
Thomas Kappeler, Klaus Pinn, and Christian Wieczerkowski.
\newblock Renormalization group flow of a hierarchical sine-{G}ordon model by
  partial differential equations.
\newblock {\em Comm. Math. Phys.}, 136(2):357--368, 1991.

\bibitem{lebleserfaty15}
Thomas Lebl\'{e} and Sylvia Serfaty.
\newblock Large deviation principle for empirical fields of log and {R}iesz
  gases.
\newblock {\em Invent. Math.}, 210(3):645--757, 2017.

\bibitem{lebleserfaty16}
Thomas Lebl\'{e} and Sylvia Serfaty.
\newblock Fluctuations of two dimensional {C}oulomb gases.
\newblock {\em Geom. Funct. Anal.}, 28(2):443--508, 2018.

\bibitem{lebowitz83}
Joel~L. Lebowitz.
\newblock Charge fluctuations in coulomb systems.
\newblock {\em Phys. Rev. A}, 27:1491--1494, Mar 1983.

\bibitem{marchettiperez89}
D.~H.~U. Marchetti and J.~Fernando Perez.
\newblock The {K}osterlitz-{T}houless phase transition in two-dimensional
  hierarchical {C}oulomb gases.
\newblock {\em J. Statist. Phys.}, 55(1-2):141--156, 1989.

\bibitem{martin88}
Ph.~A. Martin.
\newblock Sum rules in charged fluids.
\newblock {\em Rev. Modern Phys.}, 60(4):1075--1127, 1988.

\bibitem{martinyalcin80}
Ph.~A. Martin and T.~Yalcin.
\newblock The charge fluctuations in classical {C}oulomb systems.
\newblock {\em J. Statist. Phys.}, 22(4):435--463, 1980.

\bibitem{nazarovsodin11}
F.~Nazarov and M.~Sodin.
\newblock Fluctuations in random complex zeroes: asymptotic normality
  revisited.
\newblock {\em Int. Math. Res. Not. IMRN}, (24):5720--5759, 2011.

\bibitem{pastur06}
L.~Pastur.
\newblock Limiting laws of linear eigenvalue statistics for {H}ermitian matrix
  models.
\newblock {\em J. Math. Phys.}, 47(10):103303, 22, 2006.

\bibitem{peressly14}
Yuval Peres and Allan Sly.
\newblock Rigidity and tolerance for perturbed lattices, 2014.
\newblock {\it arXiv preprint arXiv:1409.4490}.

\bibitem{petzhiai98}
D\'{e}nes Petz and Fumio Hiai.
\newblock Logarithmic energy as an entropy functional.
\newblock In {\em Advances in differential equations and mathematical physics
  ({A}tlanta, {GA}, 1997)}, volume 217 of {\em Contemp. Math.}, pages 205--221.
  Amer. Math. Soc., Providence, RI, 1998.

\bibitem{radin81}
Charles Radin.
\newblock The ground state for soft disks.
\newblock {\em J. Statist. Phys.}, 26(2):365--373, 1981.

\bibitem{rougerieserfaty16}
Nicolas Rougerie and Sylvia Serfaty.
\newblock Higher-dimensional {C}oulomb gases and renormalized energy
  functionals.
\newblock {\em Comm. Pure Appl. Math.}, 69(3):519--605, 2016.

\bibitem{serfaty14}
Sylvia Serfaty.
\newblock Ginzburg-{L}andau vortices, {C}oulomb gases, and renormalized
  energies.
\newblock {\em J. Stat. Phys.}, 154(3):660--680, 2014.

\bibitem{serfatyICM}
Sylvia Serfaty.
\newblock Systems of points with coulomb interactions, 2017.
\newblock {\it arXiv preprint arXiv:1712.04095}.

\bibitem{Smale}
Steve Smale.
\newblock Mathematical problems for the next century.
\newblock In {\em Mathematics: frontiers and perspectives}, pages 271--294.
  Amer. Math. Soc., Providence, RI, 2000.

\bibitem{Sut}
Andr\'as S\"ut\ifmmode~\mbox{\H{o}}\else \H{o}\fi{}.
\newblock Crystalline ground states for classical particles.
\newblock {\em Phys. Rev. Lett.}, 95:265501, Dec 2005.

\bibitem{taovu13}
Terence Tao and Van Vu.
\newblock Random matrices: sharp concentration of eigenvalues.
\newblock {\em Random Matrices Theory Appl.}, 2(3):1350007, 31, 2013.

\bibitem{The}
Florian Theil.
\newblock A proof of crystallization in two dimensions.
\newblock {\em Comm. Math. Phys.}, 262(1):209--236, 2006.

\end{thebibliography}
\bibliographystyle{plain}

\end{document}